\documentclass[a4paper, 12pt]{amsart}
\usepackage{amssymb, amsthm, enumitem, mathtools,csquotes,caption}
\usepackage[hmargin=3.2cm, vmargin={3.8cm, 3.5cm}]{geometry}
\usepackage[english]{babel}
\usepackage[final,tracking=true,kerning=true,spacing=true,
factor=1000,stretch=20,shrink=20,babel=true]{microtype}
\microtypecontext{spacing=nonfrench}
\usepackage{hyperref}
\mathtoolsset{showonlyrefs}
\usepackage[initials, msc-links]{amsrefs}
\usepackage{tikz}
\usetikzlibrary{arrows.meta}
\definecolor{refgreen}{HTML}{019100}
\definecolor{refblue}{HTML}{1002f5}
\hypersetup{colorlinks, 
linkcolor={refgreen}, 
citecolor={refgreen}, 
urlcolor={black}}

\setlist{topsep=6pt, itemsep=.4em, after=\vspace{2pt}}
\newtheorem{thm}{Theorem}[section]
\newtheorem*{thm*}{Theorem}

\newtheorem{lem}[thm]{Lemma}
\newtheorem{prop}[thm]{Proposition}

\theoremstyle{definition}

\theoremstyle{remark}
\newtheorem{rem}[thm]{Remark}

\newcommand{\C}{{\mathbb C}}
\newcommand{\R}{{\mathbb R}}
\newcommand{\D}{\mathbb{D}}
\newcommand{\N}{\mathbb{N}}

\newcommand{\capacity}{\operatorname{Cap}}
\renewcommand{\Re}{\operatorname{Re}}
\newcommand{\extremal}{\rho}

\numberwithin{equation}{section}

\begin{document}
\date{August 13, 2026}
\title{Chebyshev polynomials on a Jordan arc}
\author[B.~Buchecker]{Benedikt Buchecker}
\author[B.~Eichinger]{Benjamin Eichinger} 
\author[O.~Rubin]{ \\ Olof Rubin}
\author[A.~Wennman]{Aron Wennman}

\begin{abstract}
We describe the asymptotics of Chebyshev polynomials on an analytic Jordan arc in the plane. 
This gives an affirmative answer to a conjecture of Christiansen--Simon--Zinchenko, 
based on predictions of Widom from 1969. The proof combines
weighted Faber polynomials with extremal signatures, discrete orthogonal
polynomials and a Marcinkiewicz--Zygmund sampling inequality, 
and yields Szegő--Widom asymptotics for the
Chebyshev polynomials themselves.
\end{abstract}

\newgeometry{hmargin=3.2cm, vmargin={3.6cm, 3cm}}
\maketitle

\section{Introduction}
% !TeX root = ../main.tex

\thispagestyle{empty}
\label{s:intro}
Given a compact set \(K\) in the plane, we denote by \(T_n\) the monic 
polynomial of degree \(n\) with minimal supremum norm on \(K\), that is,
\begin{equation}
\label{eq:Cheb}
\lVert T_n\rVert_K=\min_{a_0,\ldots,a_{n-1}}
\max_{z\in K}\Big|z^n+\sum_{k=0}^{n-1}a_kz^k\Big|.
\end{equation}
These extremal polynomials are referred to as {\em Chebyshev polynomials}, and are 
classical objects in approximation theory and complex analysis.
When \(K\subset \R\), the Chebyshev polynomials carry several useful rigidity 
properties.
One of the most important is the {\em Alternation Theorem} \cite{CSZ-Invent}*{Theorem~1.1}, 
which goes back to Borel and Markov. It states that a monic polynomial \(P\) of degree \(n\)
is the Chebyshev polynomial for \(K\) if and only if there exist points \(x_0<\ldots<x_n\)
in \(K\) such that
\[
P(x_j)=(-1)^{n-j}\lVert P\rVert_{K},\qquad 0\le j\le n.
\]
In the model case \(K=[-1,1]\), the Alternation Theorem immediately identifies 
the extremal polynomial as the classical Chebyshev polynomial of the first kind,
\begin{equation}
\label{eq:classical-Cheb}
T_n(x)=2^{1-n}\cos\big(n\arccos (x)\big),\qquad x\in [-1,1].
\end{equation}

Many authors have contributed to the literature 
on Chebyshev polynomials for subsets of the real line,  
including the now classical work of 
Chebyshev, Bernstein, Akhiezer, Rivlin, Shapiro and Widom 
\cites{Chebyshev54, Achieser, Bernstein, RivlinShapiro, Widom} to mention a few;
see \cites{SodinYuditskii} for more on the history. 
A series of decisive results were given more recently 
by Christiansen, Simon, Yuditskii and Zinchenko 
in \cites{CSZY,CSZ-Invent}. 

\restoregeometry

When leaving the real line, the alternation theorem fails, and much less is known
about both structural and asymptotic properties of \(T_n\).
One exceptional case is that of a smooth Jordan curve \(K=\gamma\), 
studied already by Faber in 1919, see \cite{Faber1919}.
The natural quantity is the {\em Widom factor}
\[
\mathcal{W}_n(\gamma)=\frac{\lVert T_n\rVert_\gamma}{\capacity(\gamma)^n},
\]
normalized by the \(n\)th power of the logarithmic capacity of \(\gamma\).
A theorem of Szeg\H{o} states that \(\mathcal{W}_n(K)\ge 1\)
holds for any compact \(K\) of positive capacity, and Faber’s work implies 
that this lower bound is asymptotically attained for any analytic Jordan curve.
Faber's results have since been extended in various directions; see e.g.\ 
\cites{Andrievskii, BEZ,Suetin, TotikVarga} as well as   
the recent survey \cite{RubinRev}. In particular,
we have
\[
\lim_{n\to\infty} \mathcal{W}_n(\gamma)=1
\]
for piecewise Dini-smooth Jordan curves with finitely
many corners and cusps \cite{MDRW}.

In his seminal 1969 paper \cite{Widom}, Harold Widom studied 
the asymptotics of extremal polynomials for
systems of disjoint Jordan curves and arcs in the plane. 
For \(L^2\)-extremals he obtained a remarkably complete picture 
in this generality, while for Chebyshev polynomials
the analysis was limited to systems of closed Jordan curves. 
Even the very natural case of a single smooth Jordan arc remained open, 
and Widom conjectured that in this setting one should have
\begin{equation}
\label{eq:Widom-conj}
\lim_{n\to\infty}\mathcal{W}_n(\gamma)= 2.
\end{equation}
That is, the Widom factor for a Jordan arc
should asymptotically be \enquote{twice as large} 
as for a Jordan curve of the same logarithmic capacity.
As noted by Totik--Yuditskii \cite{TotikYuditskii} and discussed 
in \cites{Alpan2022, Totik2014}, this conjecture is not correct as stated.
Indeed, Thiran--Detaille \cite{ThiranDetaille1991} found that for the circular arc
\[
\gamma_\alpha=\{e^{i\theta}:|\theta|\le \alpha\},\qquad 0<\alpha\le \pi,
\] 
the Widom factors converge to \(2\cos^2(\alpha/4)\), 
with values ranging from \(1\) to \(2\). 
This coincides with the square of the known limit 
of the corresponding \(L^2\)-Widom factors relative to 
equilibrium measure \cite{AlpanZinchenko}*{Theorem~5.1} and 
based on this Christiansen--Simon--Zinchenko formulated 
a revised conjecture \cite{CSZ-Rev}*{Conjecture~3.4}, which
has remained open. They predicted that
\begin{equation}
\label{eq:CSZ-conj}
\lim_{n\to\infty}\mathcal{W}_n(\gamma)= 1/\extremal(\infty),
\end{equation}
where \(\extremal\) is the outer function 
with boundary moduli \(|\extremal_\pm|=|\phi'_\pm|/(|\phi'_+| + |\phi'_-|)\)
on the two sides of \(\gamma\). Here,  \(\phi\) 
is the canonical exterior conformal map and \(\pm\) 
indicate boundary values taken from the two sides. 
This conjecture turns out to be correct.
\begin{thm}
The Christiansen--Simon--Zinchenko conjecture \eqref{eq:CSZ-conj} 
holds for any analytic Jordan arc \(\gamma\).
\label{thm:main}
\end{thm}

The boundary values \(|\extremal_\pm|\) are given as the ratio of the
one-sided harmonic measures density \(\frac{1}{2\pi}|\phi'_\pm|\,|dz|\)
to the standard two-sided harmonic measure density 
\(d\omega_{\C\setminus\gamma,\infty}=\tfrac{1}{2\pi}(|\phi'_+|+|\phi'_-|)\,|dz|\), 
relative to the point at infinity.
The quantity \(\extremal(\infty)\) can be conveniently 
expressed in terms of an entropy integral
\[
\log \extremal(\infty)=\int_{\gamma}\big[|\extremal_+|\log|\extremal_+| 
+ \big(1-|\extremal_+|\big)
\log\big(1-|\extremal_+|\big)\big]\,d\omega_{\C\setminus\gamma,\infty}.
\]
From this representation, it is not difficult 
to see that \(1/\extremal(\infty)\in (1,2]\), with
the upper extreme case occurring only for a line segment \cite{Alpan2022}. If
an arc is deformed into a closed curve by letting it {\em bite its own tail}, 
then we formally recover \(1/\extremal(\infty)=1\).

The assumption that \(\gamma\) is analytic is 
made to keep the ideas as transparent as possible, 
but we expect that it can be relaxed considerably. 
In fact, the upper bound already holds for \(C^{2+\alpha}\) arcs.
Our proof can also accommodate a continuous weight on \(\gamma\), which 
changes the asymptotics by a Szeg\H{o} function-factor. However, we will not
pursue this here.

The norm asymptotics can be strengthened to asymptotics of the Chebyshev
polynomials themselves, in the spirit of Szegő--Widom asymptotics. 
Just as above, we use the notation \(f_\pm\) 
for the two (non-tangential) boundary values 
on \(\gamma\) of a function \(f\) defined on 
the complement \(\widehat{\C}\setminus \gamma\) of \(\gamma\) 
in the Riemann sphere.
To formulate the result, let
\(g\) be the outer function \(g(z)=\extremal(z)/\extremal(\infty)\).

\begin{thm}
\label{thm:SzegoWidom}
If \(\gamma\) is an analytic Jordan arc, we have 
\[
\Big\lVert \frac{T_n}{\capacity(\gamma)^n}-
\big(g_+ \phi_+^n + g_- \phi_-^n\big)\Big\rVert_{L^2(\gamma,|dz|)}
=O\Big(\frac{\log n}{n}\Big),
\]
as well as the locally uniform Szeg\H{o}--Widom 
asymptotics on \(\widehat{\C}\setminus \gamma\)
\[
T_n(z)=\capacity(\gamma)^n
g(z)\phi^n(z)
\Big(1+O\Big(\frac{\log n}{n}\Big)\Big),
\]
as \(n\to\infty\).
\end{thm}

\subsection{Outline of the proofs}
Our proof of Theorem~\ref{thm:main} has two main parts. 
We first establish an upper bound by producing trial 
polynomials in the form of weighted Faber polynomials
\(F_n(g,z)\), defined as the polynomial part of the 
Laurent series of 
\[
g(z)\phi^n(z)= \sum_{k=-\infty}^{n} a_k(g) z^k,
\]
where the {\em weight} \(g\) lies in \(H^\infty(\C\setminus \gamma)\).
By analyzing the asymptotics of \(F_n(g,z)\) along \(\gamma\) 
we transfer the Chebyshev extremal problem to the \(n\)-independent
extremal problem of minimizing
\begin{equation}
\label{eq:gplus-gminus}
\max_{z\in\gamma}\Bigl( |g_+(z)|+|g_-(z)|\Bigr)
\end{equation}
over all sufficiently regular functions \(g\in H^\infty(\C\setminus \gamma)\) 
with \(g(\infty)=1\); cf.~\eqref{eq:Faber_asymptotics_arc}. 
Quite remarkably, this limiting extremal 
problem is explicitly solvable, and 
its minimizer is precisely the function 
\(g(z)=\extremal(z)/\extremal(\infty)\) from above,
cf.\ Theorem~\ref{thm:extremal_problem}.

To establish the matching lower bound, we use a notion 
called \textit{extremal signatures}, which dates back to 
ideas of Kolmogorov \cite{Kolmogorov}, Remez \cite{Remez}, 
Videnskii \cite{Vidensky} and Rivlin--Shapiro \cite{RivlinShapiro}.
This notion is equivalent to optimal prediction measures (OPMs) which have 
appeared in recent work on Chebyshev and residual polynomials. 
The idea is as follows. By an abstract min--max principle, one can show that  
\[
\min_{P}\max_{z\in\gamma}|P(z)|^2
=\max_{\mu\in\mathcal{M}(\gamma)}\min_{P}\int_{\gamma}|P(z)|^2d\mu(z),
\]
where \(P\) runs over all monic polynomials of fixed degree, 
and where \(\mathcal{M}(\gamma)\) denotes the space of probability measures 
supported on \(\gamma\). In particular, for any such measure \(\mu\), 
the norm of the \(n\)th monic orthogonal polynomial in \(L^2(\mu)\) 
supplies a lower bound for \(\lVert T_n\rVert_{\gamma}\).
We construct discrete measures \(\mu_n\) supported on \(n+1\) suitably chosen points,
for which the norms of the corresponding discrete orthogonal polynomials can be analyzed; 
see Theorem~\ref{thm:opm}. 
This is based on an explicit formula for discrete Chebyshev 
polynomials implicitly contained in
work by Fekete and von Neumann \cites{Fekete, FeketevonNeumann} 
and explicitly given by Motzkin and Walsh in \cite{MotzkinWalsh}.
Despite being very natural,
to the best of our knowledge the systematic construction of \(n\)-dependent
trial measures to obtain lower bounds has not
explicitly appeared in the literature, apart 
from an exactly solvable case of a different 
flavour analyzed in \cite{BosLevenberOrtegaCerda}. 
 
In the lower bound analysis we encounter 
the notion of \enquote{{\em dual Faber polynomials}} \(E_{n+1}\)
which may be of independent interest; see Section~\ref{s:dual-Faber}. 
These are polynomials of degree \(n+1\) whose zeros are precisely the extremal 
points of the weighted Faber polynomial from the 
upper-bound construction, reminiscent of the relation 
between the classical Chebyshev polynomials of the first and second kinds.
For instance, the dual polynomial associated with 
\eqref{eq:classical-Cheb} would be
\[
E_{n+1}(z)=(z^2-1)\,U_{n-1}(z),
\]
where \(\{U_k\}_{k\ge 0}\) is the sequence of monic Chebyshev polynomials of the second kind. 

To prove Theorem~\ref{thm:SzegoWidom} on Szeg\H{o}--Widom asymptotics, 
we first establish a stability result for the 
near-optimal prediction measures and their associated discrete orthogonal polynomials; see
Lemma~\ref{lem:stab}. 
One consequence is that \(\frac{T_n(z)}{\capacity(\gamma)^n}\) 
is well-approximated by the weighted Faber polynomial \(F_n(g,z)\) in a 
discrete \(L^2\)-norm supported on the zeros \(\{z_j\}_{j=0}^n\) of \(E_{n+1}\). 
This can be promoted to \(L^2\)-convergence on \(\gamma\)
using a Marcinkiewicz--Zygmund type sampling inequality
\[
\int_\gamma|P(z)|^2\,|dz|\le \frac{C(\log n)^2}{n}\sum_{j=0}^n |P(z_j)|^2,
\] 
valid for any polynomial with \(\deg P\le n\). This inequality is contained in 
Lemma~\ref{lem:MZ}, which is directly inspired by a series of works of Chui--Zhong;
see \cites{ChuiZhong, ChuiZhong-JAT}.
The locally uniform Szeg\H{o}--Widom asymptotics
then follows by standard arguments.

\subsection{Discussion}
Beyond the present application, we expect that the proof scheme developed here
could be useful in other optimal approximation problems. It provides, in this
setting, a partial replacement for the alternation and stability mechanisms that
are available on the real line but which have no direct complex counterpart.

Several interesting questions on extremal polynomials on arcs and related sets 
still remain unanswered. 
For instance, one would like to complete Widom's program by describing the limit points of
\(\mathcal W_n(\gamma)\) when \(\gamma\) consists of several disjoint arcs and curves,
and for this new ideas appear to be needed.
Another intriguing question is whether Widom's conjecture 
holds in the original form \eqref{eq:Widom-conj} if we constrain the zeros
to lie on the arc \(\gamma\). More precisely, we define a modified Widom factor
\[
\mathcal{W}_n^*(\gamma)
= \capacity(\gamma)^{-n}
\inf_{z_1,\ldots,z_n\in\gamma}
\Big\lVert 
\prod_{j=1}^n (z-z_j)\Big\rVert_{\gamma},
\]
which corresponds to minimizing the norm of 
monic polynomials with all their zeros constrained to 
the arc \(\gamma\).
Does it hold that \(\lim_{n\to\infty}\mathcal{W}_n^*(\gamma)=2\)?
The answer is clearly yes for an interval when the constraint is inactive, 
but interestingly it is also true for circular arcs \cite{MR14}, 
and we believe that the same should hold for general arcs with appropriate smoothness.
While the upper bound follows by adapting our construction, the lower bound seems to 
require new ideas\footnote{Following a talk by 
one of the authors, Zakhar Kabluchko communicated a 
proof of this conjecture for analytic arcs, 
together with a counterexample if no regularity is assumed. 
These proofs were found by ChatGPT 6 Pro (\(25/05/2026\))
and have not yet been fully digested by the authors.}.

\section{Faber polynomials on curves and arcs}
% !TeX root = ../main.tex
\label{s:faber}

One central notion to our approach is that of Faber polynomials. 
To keep this exposition self-contained, we devote this section to
a brief outline of their 
fundamental properties, in particular asymptotics.
All of this material is well-known, and some readers 
may wish to jump ahead to Section~\ref{s:upper-bound}.

Let \(K\) denote a non-trivial continuum, 
i.e., a compact connected set containing more than one point.
By the Riemann mapping theorem, 
there exists a unique conformal mapping \(\phi\)
from the unbounded component \(\Omega\) of \(\C\setminus K\) 
to the exterior disk \(\Delta\coloneqq \{w:|w|>1\}\) such that
\[\phi'(\infty) \coloneqq \lim_{z\to \infty}\frac{\phi(z)}{z}>0.\]
Note that \(\phi'(\infty) = 1/\capacity(K)\), and therefore
\[
\mathcal{W}_n(K) = \phi'(\infty)^n\|T_n\|_K.
\]
In this setting, one trivially obtains Szeg\H{o}'s inequality from
\[
\|T_n\|_K = \|T_n/\phi^n\|_K\geq \lim_{z\to\infty}\frac{T_n(z)}{\phi(z)^n} 
= \frac{1}{\phi'(\infty)^n}.
\]

If \(g\) is an analytic function on \(\Omega\) with 
\(g(\infty) = 1\), we denote by \(F_n(g,z)\) the polynomial 
part of the Laurent series of \(g(z)\phi^n(z)\) at \(\infty\). Alternatively, 
if \(C\) denotes a positively oriented Jordan curve containing 
\(K\) and the point \(z\) in its interior, then 
\begin{equation}
F_n(g,z) = \frac{1}{2\pi i}\int_{C}\frac{g(\zeta)\phi(\zeta)^n}{\zeta-z}d\zeta.
\label{eq:cauchy_integral_formula_faber}
\end{equation}
Assume now that \(\Gamma = K\) is an analytic Jordan curve. 
Then \(\phi\) extends to a conformal map on a neighborhood of \(\overline{\Omega}\).
Let us likewise assume that \(g\) has a holomorphic extension to 
a neighborhood of \(\overline{\Omega}\).
Take two positively oriented Jordan curves \(C_1\) and \(C_2\) 
contained in this neighborhood
such that \(C_1\) lies interior to \(\Gamma\) which in turn lies interior to \(C_2\). 
If \(z\) lies between \(C_1\) and \(C_2\), 
then \eqref{eq:cauchy_integral_formula_faber} implies that
\[
\begin{aligned}
g(z)\phi(z)^n &= \frac{1}{2\pi i}\int_{C_2}
\frac{g(\zeta)\phi(\zeta)^n}{\zeta-z}d\zeta-\frac{1}{2\pi i}
\int_{C_1}\frac{g(\zeta)\phi(\zeta)^n}{\zeta-z}d\zeta \\ 
& = F_n(g,z)-\frac{1}{2\pi i}\int_{C_1}
\frac{g(\zeta)\phi(\zeta)^n}{\zeta-z}d\zeta
\end{aligned}
\]
and consequently
\begin{equation}
|g(z)\phi(z)^n-F_n(g,z)|\leq 
\frac{cr^{n}}{\operatorname{dist}(z,C_1)},\quad 
r = \max_{\zeta\in C_1}|\phi(\zeta)|<1, \quad c>0.
\label{eq:Faber_curve_asymptotics}
\end{equation}
This observation forms the basis for solving 
the weighted Chebyshev problem in the 
setting of a Jordan curve.

In the setting of a Jordan arc we need to take additional care.
The essential difference for arcs is that \(\phi\) and \(g\) will 
have two sets of boundary values, one from 
each side of the arc, and the Faber polynomial
receives a contribution from each side.
To be precise, let \(\gamma = K\) denote an analytic Jordan 
arc which we may assume to have endpoints at \(\pm 1\).
The mapping \(z+\sqrt{z^2-1}\) defined on \(\Omega\) ``opens up'' the arc 
and maps \(\Omega\) to a domain \(D\) -- the exterior of an analytic 
Jordan curve containing the origin in its interior, 
see \cite{Widom}*{Lemma~11.1}. Its inverse is given by
\[
\varphi(s) = \frac{s+s^{-1}}{2}.
\]
For any \(z\in \C\), and any positively oriented Jordan curve \(C\) 
containing \(\gamma\) and \(z\) in its interior, we obtain 
\[
F_n(g,z)=\frac{1}{2\pi i}
\int_{C}\frac{g(\zeta)\phi^n(\zeta)}{\zeta-z}d\zeta 
= \frac{1}{2\pi i}\int_{\varphi^{-1}(C)}\frac{g\circ \varphi(s)
(\phi\circ\varphi)^n(s)}{\varphi(s)-z}\varphi'(s)ds.
\]
Since \(\varphi:D\to \Omega\) is one-to-one with a 
two-to-one extension \(\partial D\to\gamma\) away 
from the endpoints we can find a \(\sigma\in \overline{D}\) 
such that \(\varphi(\sigma) = z\). 
Using that
\[
\frac{\varphi'(s)}{\varphi(s)-\varphi(\sigma)}
=\frac{s-s^{-1}}{(s-\sigma)(s-\sigma^{-1})} 
= \frac{1}{s-\sigma}+\frac{1}{s\sigma (s-\sigma^{-1})}
\]
we obtain
\begin{align}
\notag F_n(g,z) & 
= \frac{1}{2\pi i}\int_{\varphi^{-1}(C)}
\frac{g\circ \varphi(s)(\phi\circ\varphi)^n(s)}{s-\sigma}ds
+\frac{1}{2\pi i}\int_{\varphi^{-1}(C)}
\frac{g\circ \varphi(s)(\phi\circ\varphi)^n(s)}{s\sigma(s-\sigma^{-1})}ds\\
& = \frac{1}{2\pi i}\int_{\varphi^{-1}(C)}\frac{G(s)\Phi^n(s)}{s-\sigma}ds
+\frac{1}{2\pi i\sigma}\int_{\varphi^{-1}(C)}
\frac{G(s)\Phi^n(s)}{s}\frac{ds}{s-\sigma^{-1}},
\label{eq:transfered_faber}
\end{align}
where the composition \(\Phi\coloneqq \phi\circ\varphi\) 
is now simply the canonical conformal map from 
\(D\) to \(\Delta\) and \(G\coloneqq g\circ\varphi\).

\begin{figure}[ht!]
\centering
\includegraphics[width = 0.8\textwidth]{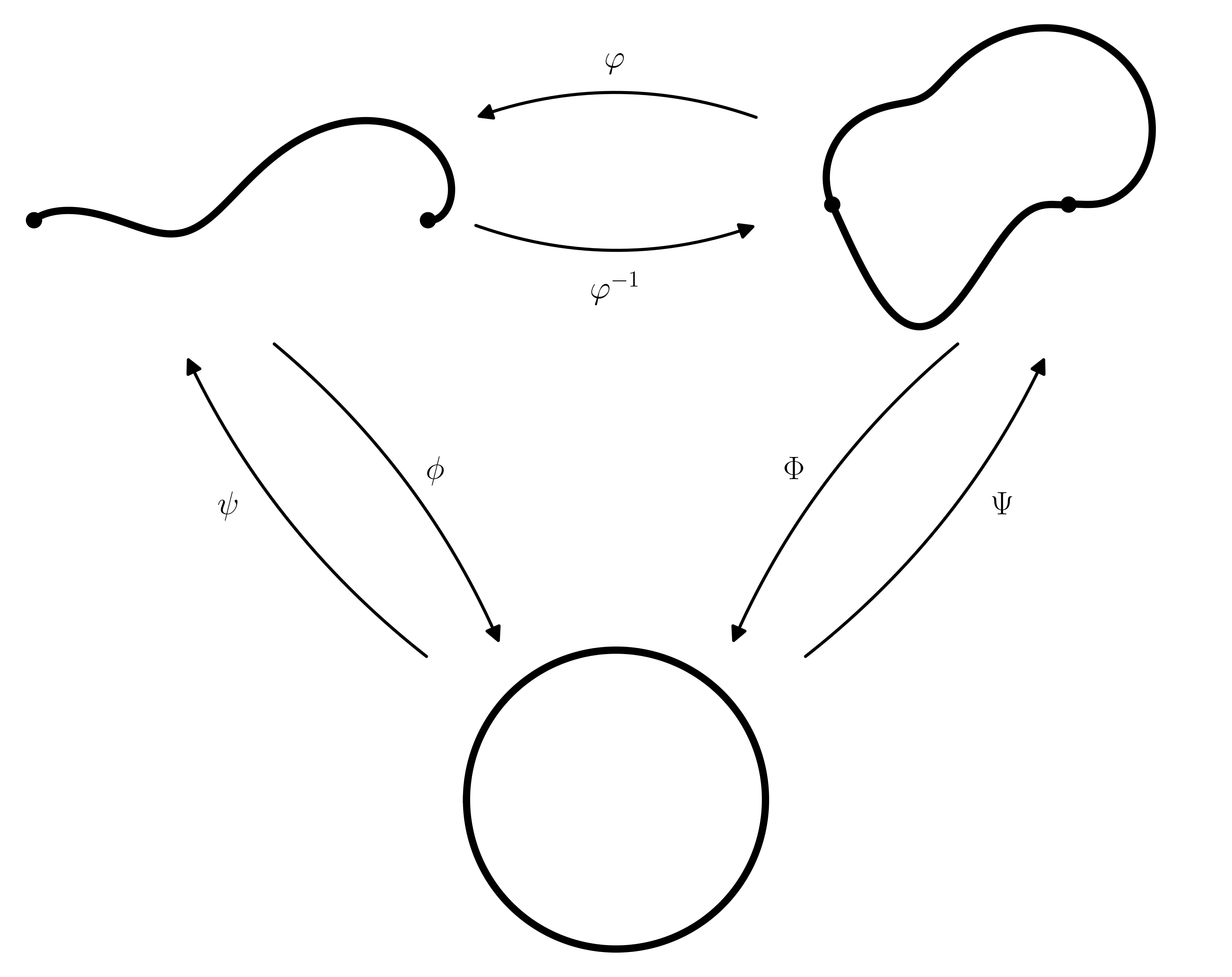}
\captionsetup{width=.95\linewidth}
\caption{Diagram of the various exterior 
conformal mappings used throughout the paper.}
\label{fig:maps}
\end{figure}

We recognize the right-hand side of \eqref{eq:transfered_faber} 
as (essentially) the sum of two Faber polynomials on \(D\) with respect to the 
weights \(G\) and \(G/s\) evaluated at \(\sigma\) and \(\sigma^{-1}\) 
respectively. That is, if the Faber polynomials on \(D\) are denoted by
\(F_{n,D}\), then we have
\[
F_n(g,z)=F_{n,D}(G,\sigma)+\frac{1}{\sigma}F_{n,D}\big(G/s,\sigma^{-1}\big).
\] 
Suppose next that \(G\) extends analytically to a neighborhood 
of \(\overline{D}\). Then, since \(\sigma\) is close to \(\partial D\) 
if and only if \(\sigma^{-1}\) is close to \(\partial D\), the asymptotics 
\eqref{eq:Faber_curve_asymptotics} shows that the right-hand 
side of \eqref{eq:transfered_faber} equals
\begin{align*}
F_n(g,z)=G(\sigma)\Phi^n(\sigma)
+G(\sigma^{-1})\Phi^n(\sigma^{-1})+O(r^n) 
\end{align*}
for some \(r<1\).
In other words, if \(z\) is sufficiently close to \(\gamma\), 
then we obtain
\[
F_n(g,z) = G\circ\varphi^{-1}(z)\Phi^n\circ \varphi^{-1}(z)
+G(1/\varphi^{-1}(z))\Phi^n(1/\varphi^{-1}(z))+O(r^n).
\]
Observe that \(\varphi^{-1}(z)\in D\) while 
\(1/\varphi^{-1}(z)\) lies inside the contour \(\partial D\) 
and therefore these evaluations are interpreted in the sense 
of analytic continuation. The property that the arc may be 
approached from two sides manifests in this formula.
Introducing the labels \(g_+\) and \(g_-\) for the boundary 
values of the function \(g\) as we approach from either 
side we find that
\begin{equation}
F_n(g,z) = g_+(z)\phi_+(z)^n+g_-(z)\phi_-(z)^n+O(r^n),
\label{eq:Faber_asymptotics_arc}
\end{equation}
which holds uniformly for \(z\in \gamma\) but also 
in a neighborhood of \(\gamma\) if the formula is 
interpreted through analytic extension. 

\begin{rem}
\label{rem:reg-prel}
We conclude this section with a regularity observation about the outer function
\(\extremal\) introduced in Section~\ref{s:intro}. If \(\gamma\) is of class
\(C^{1+\alpha}\), where \(0<\alpha<1\), then the canonical conformal map
\(\Phi=\phi\circ\varphi\) has a non-vanishing H\"older-\(\alpha\) derivative
on \(\overline D\). On \(\partial D\),
the boundary modulus of \(\extremal\circ\varphi\) is given by
\begin{equation}
\big|\extremal\circ\varphi(\sigma)\big|
=\frac{|\Phi'(\sigma)|}{|\Phi'(\sigma)|+
\left|\frac{d}{d\sigma}\Phi(\sigma^{-1})\right|},
\qquad \sigma\in\partial D.
\label{eq:boundary_data}
\end{equation}
It follows that
\(\extremal\circ\varphi\) is H\"older-\(\alpha\) continuous on
\(\overline D\). If \(\gamma\) is analytic, then the right-hand side of
\eqref{eq:boundary_data} is real analytic, and
\(\extremal\circ\varphi\) extends to be holomorphic and non-vanishing on a
neighborhood of \(\overline D\). In particular, 
if we put \(g=\extremal/\extremal(\infty)\), 
the asymptotic formula \eqref{eq:Faber_asymptotics_arc} 
applies for \(F_n(g,z)\).
\end{rem}

\section{The upper bound}
% !TeX root = ../main.tex

\label{s:upper-bound}

In this section we prove that
\begin{equation}
\limsup_{n\to\infty}\mathcal{W}_n(\gamma)
\leq \frac{1}{\extremal(\infty)}
\label{eq:upper_bnd}
\end{equation}
where \(\extremal\) is the outer function with boundary modulus given by
\begin{equation}
\label{eq:extremal-mod}
|\rho_\pm(z)|=\frac{|\phi'_\pm(z)|}{|\phi'_+(z)|+|\phi'_-(z)|},\qquad z\in\gamma.
\end{equation}
We do this by explicitly constructing trial polynomials of the form \(F_n(g,z)\),
where \(g\) is chosen among \(H^\infty\)-functions on
\(\C\setminus\gamma\) with \(g(\infty)=1\), so as to minimize the norm
\(\lVert F_n(g,\cdot)\rVert_{\gamma}\) asymptotically.
Since \(F_n(g,z)\) has degree \(n\) and leading coefficient
\(\phi'(\infty)^n\), we immediately find that
\[
\mathcal{W}_n(\gamma)=\phi'(\infty)^n\|T_n\|_\gamma
\leq \|F_n(g,\cdot)\|_\gamma.
\]
Hence, our claimed upper bound will be achieved if we can arrange that
\[
\|F_n(g,\cdot)\|_{\gamma}\longrightarrow \frac{1}{\extremal(\infty)}.
\]

Assume that \(\gamma\) is a Jordan arc of class \(C^{2+\alpha}\) 
for some \(\alpha>0\), and introduce the family
\[
\mathcal{H}=\big\{g\in H^\infty(\Omega): g(\infty)=1,\; g\circ\varphi
\in C^\eta(\overline D)\text{ for some }\eta>0\big\}.
\]
For \(g\in \mathcal{H}\) we gather from \cite{Widom}*{Lemma~11.2} 
that there exists some \(\varepsilon>0\), 
depending on the regularity of \(\gamma\) and \(g\),
such that
\begin{equation}
\label{eq:widom_asymptotics}
\max_{z\in \gamma}
\left|F_n(g,z)-\Bigl(g_+(z)\phi_+(z)^n+g_-(z)\phi_-(z)^n\Bigr)\right| 
= O\big(n^{-\varepsilon}\big),\quad n\to \infty,
\end{equation}
see also \eqref{eq:Faber_asymptotics_arc}. 
As we will now explain, this formula can be used to show the following.

\begin{thm}
\label{thm:optimal_faber}
Let \(\gamma\) be a \(C^{2+\alpha}\) Jordan arc where \(0<\alpha<1\).
Then
\[
\inf_{g\in \mathcal{H}}\lim_{n\to\infty}\|F_n(g,\cdot )\|_\gamma 
=\frac{1}{\extremal(\infty)}.
\] 
\end{thm}

By expanding the absolute value in \eqref{eq:widom_asymptotics} 
and using that \(\phi\) has modulus \(1\) on the arc, we find that
\begin{equation}
|F_n(g,z)|^2 = |g_+(z)|^2+|g_-(z)|^2
+2\Re\left(g_+(z)\phi_+(z)^n\overline{g_-(z)\phi_-(z)^n}\right)
+O\big(n^{-\varepsilon}\big).
\label{eq:faber_pol_abs}
\end{equation}
We introduce the functions
\begin{equation}
\theta(z) \coloneqq \arg\left(\phi_+(z)\overline{\phi_-}(z)\right),
\quad \beta(z)\coloneqq \arg\left(g_+(z)\overline{g_-(z)}\right).
\end{equation}
The latter quantity will only be considered when both boundary values 
\(g_\pm\) are non-vanishing,
in which case there is no ambiguity in the definition.
The image of $\theta$ is $[0,2\pi]$ and this angular function 
is strictly increasing as we traverse $\gamma$ between its endpoints. 
We may rewrite \eqref{eq:faber_pol_abs} as
\begin{equation}
|F_n(g,z)|^2 = |g_+(z)|^2+|g_-(z)|^2
+2|g_+(z)||g_-(z)|\cos\bigl(n\theta(z)+\beta(z)\bigr)
+O\big(n^{-\varepsilon}\big).
\label{eq:faber_pol_abs_cos}
\end{equation}
This provides us with a very useful heuristic; the Faber polynomial
oscillates rapidly between an upper and a lower envelope, given by
\[
|g_+(z)|+|g_-(z)|\quad\text{and}\quad \big|\,|g_+(z)|-|g_-(z)|\,\big|,
\]
respectively. In particular, we may deduce that
\[
\limsup_{n\to\infty}\max_{z\in \gamma}|F_n(g,z)|\leq 
\max_{z\in \gamma}\Bigl(|g_+(z)|+|g_-(z)|\Bigr),
\]
and in fact equality holds.
\begin{lem}
\label{lem:oscillatory}
If \(g\in \mathcal{H}\), then
\[
\lim_{n\to\infty}\|F_n(g;\cdot)\|_\gamma 
= \max_{z\in \gamma}\Bigl(|g_+(z)|+|g_-(z)|\Bigr).
\]
\end{lem}

\begin{proof}
The upper bound follows immediately from \eqref{eq:widom_asymptotics}.
For the reverse inequality, set
\[
M=\max_{z\in\gamma}\bigl(|g_+(z)|+|g_-(z)|\bigr),
\]
and choose \(z_0\in\gamma\) where this maximum is attained. If
\(g_+(z_0)g_-(z_0)=0\), then \eqref{eq:faber_pol_abs} gives
\[
|F_n(g,z_0)|=M+o(1),
\]
and there is nothing left to prove.

Suppose therefore that both boundary values are non-vanishing at \(z_0\), and let
\(\epsilon>0\). By continuity, there is a non-degenerate subarc \(I\)
containing \(z_0\) on which
\[
|g_+(z)|+|g_-(z)|>M-\epsilon
\]
and on which \(\beta\) is continuous. Since \(\theta\) is strictly increasing,
the image of \(I\) under \(z\mapsto\theta(z)+\beta(z)/n\) contains a fixed
non-degenerate interval for all sufficiently large \(n\). Consequently, the
image of \(I\) under \(z\mapsto n\theta(z)+\beta(z)\) contains a multiple of
\(2\pi\). At such a point \(z_n\in I\), formula
\eqref{eq:faber_pol_abs_cos} gives
\[
|F_n(g,z_n)|^2=\bigl(|g_+(z_n)|+|g_-(z_n)|\bigr)^2+o(1)
\geq (M-\epsilon)^2+o(1).
\]
The conclusion follows by first letting \(n\to\infty\) and then
\(\epsilon\to 0\).
\end{proof}

We are led to the extremal problem of minimizing
\begin{equation}
\max_{z\in \gamma}\Bigl(|g_+(z)|+|g_-(z)|\Bigr),
\label{eq:extremal_problem}
\end{equation}
within the class \(\mathcal{H}\). 

\begin{thm}
\label{thm:extremal_problem}
The extremal problem \eqref{eq:extremal_problem} has a unique solution
within the class \(\mathcal{H}\), given by 
\(g(z)=\extremal(z)/\extremal(\infty)\),
and the corresponding minimal value equals
\[ 
\max_{z\in\gamma}\, \Bigl(|g_+(z)|+|g_-(z)|\Bigr)
=\frac{1}{\extremal(\infty)}.
\]
\end{thm}
\begin{rem}
We first give a heuristic derivation, which explains why the function 
\(g\) solves the above extremal problem. 
First, it can be shown that \(|g_+|+|g_-|\) must be constant
on \(\gamma\); denote this constant by \(M\). 
For a minimizer \(g\), we put \(h=g/M\) and notice that the function
\(h\) maximizes \(h(\infty)\) among $H^\infty$-functions with 
\(|h_+|+|h_-|=1\) on \(\gamma\). 
In looking for the largest possible
value of \(h(\infty)\), we may without loss of generality assume that \(h\)
is outer. We then have that
\[
\log h(\infty)=\int_\gamma
\left[\frac{|\phi_+'|}{|\phi_+'|+|\phi_-'|}\log |h_+|
+\frac{|\phi_-'|}{|\phi_+'|+|\phi_-'|}\log(1-|h_+|)\right]
d\omega_{\C\setminus\gamma,\infty}.
\]
A simple variational argument shows that 
the integrand is pointwise maximal when
\[
|h_+|=\frac{|\phi_+'|}{|\phi_+'|+|\phi_-'|}.
\]
Thus, the sought outer function is \(h=\extremal\), and passing back to \(g\) 
by undoing the renormalization suggests that 
the solution of \eqref{eq:extremal_problem} is
\(g=\extremal/\extremal(\infty)\).
\end{rem}

We proceed with a rigorous proof of this observation.

\begin{proof}[Proof of Theorem~\ref{thm:extremal_problem}]
Assume that \(g\in\mathcal H\). Since harmonic measure is a
probability measure,
\begin{align}
\int_\gamma\bigl(|g_+(z)|+|g_-(z)|\bigr)
d\omega_{\C\setminus\gamma,\infty}(z)
\leq\max_{z\in\gamma}\bigl(|g_+(z)|+|g_-(z)|\bigr).
\label{eq:holder_extremal_problem}
\end{align}
Let \(\psi=\phi^{-1}\). By the definition of \(\extremal\), the two
boundary contributions on the arc combine to give
\begin{align}
\begin{split}
\int_\gamma\bigl(|g_+(z)|+|g_-(z)|\bigr)
d\omega_{\C\setminus\gamma,\infty}(z)
&=\lim_{r\downarrow1}\int_0^{2\pi}
\left|\frac{g(\psi(re^{it}))}{\extremal(\psi(re^{it}))}\right|
\frac{dt}{2\pi}\\
&\geq\frac{g(\infty)}{\extremal(\infty)}
=\frac{1}{\extremal(\infty)},
\end{split}
\label{eq:max_principle_extremal_problem}
\end{align}
where the last step follows by applying the sub-mean 
value inequality to the modulus of
\(g/\extremal\), which is subharmonic.

On the other hand, the function
\[
g(z)=\frac{\extremal(z)}{\extremal(\infty)}
\]
belongs to \(\mathcal H\) by the regularity observation at the end of
Section~\ref{s:faber}, and its boundary values satisfy
\[
|g_+(z)|+|g_-(z)|
=\frac{1}{\extremal(\infty)}.
\]
This proves extremality.

The same argument also gives uniqueness. Indeed, if \(g\in\mathcal H\) is a
minimizer, then equality holds throughout the chain of inequalities above.
That is, if we put
\[
u(z)=\frac{g(z)}{\extremal(z)},
\]
then
\[
\lim_{r\downarrow1}\int_0^{2\pi}|u(\psi(re^{it}))|\,\frac{dt}{2\pi}=|u(\infty)|.
\]
Since \(|u\circ\psi|\) is subharmonic in \(\Delta\), equality in the
mean-value inequality forces \(|u\circ\psi|\) to be harmonic. However, the
modulus of a holomorphic function is harmonic only if the function is
constant. Therefore
\[
u\equiv u(\infty)=\frac{1}{\extremal(\infty)},
\]
and hence
\[
g(z)=\frac{\extremal(z)}{\extremal(\infty)}. \qedhere
\]
\end{proof}

Combining Lemma \ref{lem:oscillatory} with Theorem \ref{thm:extremal_problem} 
proves Theorem \ref{thm:optimal_faber} and in turn the upper bound \eqref{eq:upper_bnd}.

\begin{rem}
Let us take a moment to discuss the link to the corresponding 
\(L^2\)-problem investigated and solved by Widom. 
On the one hand, temporarily introducing the \(L^2\)-Widom factors
\[
\mathcal{W}_{2,n}(\gamma) = \inf_{\substack{P\text{ monic}\\ \deg P =n}}
\frac{\|P\|_{L^2(d\omega_{\mathbb{C}\setminus \gamma,\infty})}}{\capacity(\gamma)^n},
\]
it follows from Theorem~12.3 in \cite{Widom} that
\[
\lim_{n\to\infty}\mathcal{W}_{2,n}(\gamma)^2 
= \inf_{\substack{f\in H^2(\mathbb{C}\setminus \gamma)\\ f(\infty) = 1}}
\int_{\gamma}\Bigl(|f_+(z)|^2+|f_-(z)|^2\Bigr)
d\omega_{\mathbb{C}\setminus \gamma,\infty}(z),
\]
and it is clear that this extremal problem is related to 
\[
\inf_{\substack{f\in H^1(\mathbb{C}\setminus \gamma)\\f(\infty)=1}}
\int_{\gamma}\Bigl(|f_+(z)|+|f_-(z)|\Bigr)
d\omega_{\mathbb{C}\setminus \gamma,\infty}(z)
\] 
by simply taking the square root of the (never vanishing) extremizer for the latter.
On the other hand, our proof amounts to establishing
\[
\lim_{n\to\infty}\mathcal{W}_n(\gamma) 
= \inf_{\substack{f\in H^\infty(\mathbb{C}\setminus \gamma)\\ f(\infty)=1}} 
\operatorname{ess\, sup}_{z\in \gamma}\Bigl(|f_+(z)|+|f_-(z)|\Bigr)
\] 
and the proof of Theorem \ref{thm:extremal_problem} 
shows that the extremal problem on the right-hand side 
as well as the quantity
\[
\inf_{\substack{f\in H^1(\mathbb{C}\setminus \gamma)\\f(\infty)=1}}
\int_{\gamma}|f_+(z)|+|f_-(z)|d\omega_{\mathbb{C}\setminus \gamma,\infty}(z)
\] 
have the same extremal function, namely \(g=\rho/\rho(\infty)\).
The corresponding \(H^2\)-extremizer is therefore simply \(\sqrt g\).
This observation explains why the conjecture holds: the \(H^2\) and \(H^\infty\)-extremal 
problems which describe the limiting behavior of the two Widom factors  
are ultimately equivalent.
It is not clear, however, how to extend this correspondence to \(L^q\)-minimizing 
polynomials for values of \(q\) other than \(2\) and \(\infty\).
\end{rem}

\section{The lower bound}
% !TeX root = ../main.tex

\subsection{Extremal signatures and optimal prediction measures}
\label{s:optimal-prediction}

\noindent To prove Theorem \ref{thm:main} and thus give an 
affirmative answer to the Christiansen--Simon--Zinchenko 
conjecture, it remains to show that
\begin{equation}
\liminf_{n\to\infty}\mathcal{W}_n(\gamma) \geq  \frac{1}{\extremal(\infty)}
\label{eq:lower_bnd}
\end{equation}
whenever \(\gamma\) is an analytic Jordan arc.
We will base our approach on \textit{extremal signatures}, or equivalently
\textit{optimal prediction measures}; see, e.g.,~\cite{Rivlin}.
Given a compact set \(K\) containing at least \(n+1\) 
points we now temporarily introduce the notation \(T_n^K\) 
for the monic polynomial of degree \(n\) that minimizes 
the supremum norm \(\|\cdot\|_K\) over \(K\).  
We use the notation \(\mathcal{M}(K)\) for the space of 
probability measures with support contained 
in \(K\) and \(\mathcal{P}_{k}\) for the collection of 
polynomials of degree at most \(k\). 
If \(\mu\in \mathcal{M}(K)\), then trivially
\begin{equation}
\label{eq:OPM_inequality}
\inf_{p\in \mathcal{P}_{n-1}}\left(\int\left|z^n-p(z)\right|^2d\mu(z)\right)^{1/2}
\leq \left(\int |T_n^K(z)|^2d\mu(z)\right)^{1/2}\leq \|T_n^K\|_K.
\end{equation}
However, one can show that
\[
\max_{\mu\in \mathcal{M}(K)}\inf_{p\in \mathcal{P}_{n-1}}
\left(\int\left|z^n-p(z)\right|^2d\mu(z)\right)^{1/2} = \|T_n^K\|_K.
\]
and any measure \(\mu_n\) which realizes 
this maximum is called an optimal prediction measure.
This idea will be the basis of our approach, however, 
we will bypass any abstract proof concerning maximizing 
measures by simply providing an explicit solution in the 
case where the set \(K\) consists of \(n+1\) points. 

\begin{thm}
\label{thm:opm}
Let $X = \{z_0,\dotsc,z_{n}\}$ be a set of \(n+1\) distinct points. Writing 
$E_{n+1}(z) = \prod_{k=0}^{n}(z-z_k)$ we have that
\begin{equation}
\frac{T_n^X(z)}{\|T_n^X\|_X}= \sum_{k=0}^{n}
\frac{1}{|E_{n+1}'(z_k)|}\frac{E_{n+1}(z)}{(z-z_k)}.
\label{eq:explicit_cheb}
\end{equation}
In particular 
\[
\|T_n^X\|_X = \left(\sum_{k=0}^{n}\frac{1}{|E_{n+1}'(z_k)|}\right)^{-1}.
\]
An optimal prediction measure is given by
\begin{equation}
d\mu(z) = \sum_{k=0}^{n}\frac{1}{|E_{n+1}'(z_k)|}\delta_{z_k}(z)\bigg/
\sum_{k=0}^{n}\frac{1}{|E_{n+1}'(z_k)|}.
\label{eq:OPM}
\end{equation}
\end{thm}

The formula \eqref{eq:explicit_cheb} appears in 
\cite{SmirnovLebedev}*{p.~444} as a special case of Videnskii's more general result 
\cite{Vidensky}. Their proof uses discrete orthogonality in the same way that we do 
below. However, \eqref{eq:explicit_cheb} also follows directly from Lagrange 
interpolation in a similar manner to what is used already by Fekete and von Neumann in 
\cite{FeketevonNeumann}*{Section~5} to solve a related minimization problem. It can be 
deduced from Fekete's representation of nonvanishing extremal polynomials 
\cite{Fekete}, and follows as a special case of the weighted discrete Chebyshev 
theorem of Motzkin and Walsh \cite{MotzkinWalsh}.

\begin{proof}
Let us consider \(n\) as fixed and write \(E = E_{n+1}\). 
Write $Q(z)$ for the right-hand side of \eqref{eq:explicit_cheb}. 
Then $Q$ is a polynomial of degree $n$ and furthermore
\[
Q(z_k) = \frac{E'(z_k)}{|E'(z_k)|}
\]
and so $Q$ has modulus $1$ on $X$. 
We will show that \(Q\) is an orthogonal polynomial with respect
to the measure \(d\mu\) from \eqref{eq:OPM}.

Let $r_0>0$ be chosen such that $|z_k|\leq r_0$ for all $k = 0,\dotsc,n$. 
Then if $p\in \mathcal{P}_{n-1}$ we find by the Residue theorem that
\[
\frac{1}{2\pi i}\int_{|z| = r_0}\frac{p(z)}{E(z)}dz 
= \sum_{k=0}^{n}\frac{p(z_k)}{E'(z_k)}.
\]
On the other hand by taking $r>r_0$
\[
\left|\frac{1}{2\pi i}\int_{|z| = r_0}\frac{p(z)}{E(z)}dz\right| 
= \left|\frac{1}{2\pi i}\int_{|z| = r}\frac{p(z)}{E(z)}dz\right|
\leq \frac{O(r^{n})}{r^{n+1}+O(r^{n})}\rightarrow 0
\]
as $r\rightarrow \infty$. This shows that
\[
\sum_{k=0}^{n}\frac{p(z_k)}{E'(z_k)} = 0
\]
for every $p\in \mathcal{P}_{n-1}$. From the fact that
\[
\frac{1}{E'(z_k)} = \frac{\overline{E'(z_k)}}{|E'(z_k)|^2} 
= \frac{1}{|E'(z_k)|}\overline{Q(z_k)}
\]
we conclude that
\[
0 = \sum_{k=0}^{n}\frac{1}{E'(z_k)}p(z_k) 
= \sum_{k=0}^{n}\frac{1}{|E'(z_k)|}\overline{Q(z_k)}p(z_k).
\]
This shows that \(Q\) is an orthogonal polynomial of degree \(n\)
with respect to the measure \(d\mu\) defined in \eqref{eq:OPM}. 
Let us define
\[
T(z) = Q(z)\bigg/\sum_{k=0}^{n}\frac{1}{|E'(z_k)|}
\]
so that \(T\) is monic. 
From the fact that
\[\|Q\|_X = 1 = \left(\int |Q(z)|^2d\mu\right)^{1/2}\]
we conclude
\[
\|T\|_X=\left(\int |T(z)|^2d\mu(z)\right)^{1/2} 
\leq \left(\int |T_n^{X}(z)|^2d\mu(z)\right)^{1/2}\leq \|T_n^X\|_X.
\]
Uniqueness of the \(n\)th Chebyshev polynomial on \(X\)
establishes the following facts
\begin{itemize}[leftmargin=.8cm]
\item[{\rm (i)}] \(T = T_n^X\),
\item[{\rm (ii)}] equality holds in \eqref{eq:OPM_inequality} 
so \(\mu\) is an optimal prediction measure.\qedhere
\end{itemize}
\end{proof}
Note that if \(X = \{z_0,\dotsc,z_{n}\}\subset \gamma\) then
\[
\|T_n^X\|_X\leq \|T_n^\gamma\|_\gamma
\]
and more importantly, after renormalizing by capacity, 
\begin{equation}
\label{eq:cor-opm}
\mathcal{W}_n(\gamma)\ge 
\left(\sum_{k=0}^{n}\frac{1}{|\mathcal E_{n+1}'(z_k)|}\right)^{-1}\hspace{-12pt},
\qquad \mathcal E_{n+1}(z)=\phi'(\infty)^n \prod_{j=0}^n (z-z_j),
\end{equation}
and so we have obtained a strategy for deriving lower bounds.
Our aim now is to choose the point set \(X\) in an intelligent way.
Note that we have absorbed the factor \(\phi'(\infty)^n=\capacity(\gamma)^{-n}\)
into the capacity-normalized polynomial \(\mathcal E_{n+1}\),
whereas \(E_{n+1}\) in Theorem~\ref{thm:opm} is monic.

\subsection{The dual Faber polynomials}
\label{s:dual-Faber}
Since we aim to show that the upper bound from Section~\ref{s:upper-bound} is sharp, 
the Chebyshev polynomial \(T_n\) ought to be well approximated by the Faber polynomial
\(F_n(g,z)\).
In light of this, Theorem~\ref{thm:opm} suggests that the zeros of
\(\mathcal{E}\) should be placed at (or very near) the extremal points 
of the main asymptotic contribution
\[
g_+(z)\phi_+(z)^n+g_-(z)\phi_-(z)^n,\qquad z\in\gamma
\]
to \(F_n(g,z)\).
Here, \(g\) is the extremal function from the previous section given by
\[
g(z)=\frac{\extremal(z)}{\extremal(\infty)},
\]
where \(\extremal\) is the outer function from
Theorem~\ref{thm:extremal_problem}.
We will accomplish this by constructing a kind of \enquote{dual} Faber polynomial. 
Before we proceed with the construction, we discuss some
preliminaries on the regularity of \(g\) and the phase functions
\[
\theta(z)=\arg\Bigl(\phi_+(z)\overline{\phi_-(z)}\Bigr),
\qquad
\beta(z)=\arg\Bigl(g_+(z)\overline{g_-(z)}\Bigr).
\]

Recall that the map
\(\varphi^{-1}(z) = z+\sqrt{z^2-1}\) 
takes \(\Omega\) to the domain \(D\), lying to the exterior of the 
Jordan curve \(\partial D\), where \(\varphi\) is given by
\[
\varphi(s) = \frac{s+s^{-1}}{2}.
\]
The reader may consult Figure~\ref{fig:maps} for for notation related to other
conformal mappings appearing below.

\begin{lem}
\label{lem:analyticity-g}
If \(\gamma\) is an analytic Jordan arc, then
\[
G(\sigma)\coloneqq g\circ\varphi(\sigma)
\]
extends to be analytic on a neighborhood of \(\overline{D}\).
\end{lem}

\begin{proof}
This follows directly from the regularity observation in Remark~\ref{rem:reg-prel}, 
since \(g=\extremal/\extremal(\infty)\).
\end{proof}

\begin{lem}
\label{lem:extreme-pts}
For all sufficiently large \(n\), the function \(n\theta(z)+\beta(z)\)
increases strictly from \(0\) to \(2\pi n\) as \(z\) traverses the arc
from \(1\) to \(-1\). Consequently, there are precisely \(n+1\) extremal
points \(\{z_j\}_{j=0}^n=\{z_{j,n}\}_{j=0}^n\), determined by
\[
n\theta(z_j)+\beta(z_j)=2\pi j,\qquad j=0,\dotsc,n.
\]
Moreover, there exists a constant \(c>0\), depending only on \(\gamma\), 
such that
\[
|z_j-z_k|\geq \frac{c}{n^2},\qquad j\neq k.
\]
\end{lem}

Below in Appendix~\ref{s:geom-lemmas}, we will prove a refined separation claim,
stating that \(|z_j-z_{j-1}|\asymp (1+\min\{j,n-j\})/n^2\). However, for
the lower bound this simple assertion suffices.

\begin{proof}
Since
\(|g_+(z)|+|g_-(z)|=1/\extremal(\infty)\),
on \(\gamma\), the triangle inequality shows that
the modulus of
\[
g_+(z)\phi_+(z)^n+g_-(z)\phi_-(z)^n
\]
attains its maximal value \(1/\extremal(\infty)\) precisely when
\(n\theta(z)+\beta(z)\in 2\pi\mathbb Z\).
We fix a real-analytic parametrization \(\zeta:[0,\tau]\to\partial D\) 
of the part of \(\partial D\) joining \(1\) to 
\(-1\) with positive orientation, and put
\begin{equation}
\label{eq:def-AB}
\begin{aligned}
A(t)&=\arg\Phi(\zeta(t))-\arg\Phi(\zeta(t)^{-1}), \\
B(t)&=\arg G(\zeta(t))-\arg G(\zeta(t)^{-1}).
\end{aligned}
\end{equation}
Since \(\Phi\) is conformal on a neighborhood of \(D\), 
the derivative \(A'\) is positive and bounded away
from zero on \([0,\tau]\), while Lemma~\ref{lem:analyticity-g} shows that
\(B'\) is bounded. 
It follows that \((nA+B)'\asymp n\) for all sufficiently large \(n\).
Moreover, we have that
\[
A(0)=B(0)=B(\tau)=0, \qquad A(\tau)=2\pi
\]
and hence \(nA+B\) maps \([0,\tau]\) smoothly and monotonically onto \([0,2\pi n]\). 
We let \(t_j\) be the points for which \(nA(t_j)+B(t_j)=2\pi j\).
It is readily seen that \(z_j=\varphi(\zeta(t_j))\) are the desired extreme points.
Indeed,
\[
\begin{aligned}
n\theta(z_j)+\beta(z_j)&= n\arg\big(\phi_+(z_j)/\phi_-(z_j)\big) 
+ \arg\big(g_+(z_j)/g_-(z_j)\big)
\\ &
=n\arg\big(\Phi(\zeta(t_j))/\Phi(\zeta(t_j)^{-1})\big)
+ \arg\big(G(\zeta(t_j))/G(\zeta(t_j)^{-1})\big) \\
&= nA(t_j)+B(t_j)=2\pi j,\quad 0\le j\le n.
\end{aligned}
\]

Since \(A\) is regular and has non-vanishing derivative and since \(B\) is regular, 
it holds that
\[
|t_j-t_k|\geq c_1|j-k|/n.
\]
Since \(\zeta\) is regular, the same lower bound holds for
\(|\zeta(t_j)-\zeta(t_k)|\). Finally, the inverse Joukowski map is
H\"older continuous up to \(\gamma\) with parameter \(\alpha=\frac12\), so
\[
|\zeta(t_j)-\zeta(t_k)| \leq C|z_j-z_k|^{1/2}.
\]
This gives \(|z_j-z_k|\geq c/n^2\) whenever \(j\neq k\).
\end{proof}

We now look for a polynomial whose zeros are asymptotically placed at these
points. It turns out that the natural dual Faber polynomial is
\begin{equation}
\begin{aligned}
\label{eq:dual_faber}
\mathcal{E}_0(z)&=(z^2-1)F_{n-1}\Bigl(f\phi/\sqrt{\zeta^2-1},z\Bigr) \\
&=\frac{z^2-1}{2\pi i}
\int_C \frac{f(\zeta)\phi(\zeta)^n}{\sqrt{\zeta^2-1}}\frac{d\zeta}{\zeta-z}.
\end{aligned}
\end{equation}
The intuition is that the square-root singularity switches the sign of the
interaction between the two branches, changing the oscillations from a cosine
pattern to a sine pattern.
To formulate this precisely, let \(z=\varphi(\sigma)\). For a function \(h\) such
that \(H=h\circ\varphi\) has an analytic extension to a neighborhood of
\(\overline D\), we use the notation
\begin{equation}
\label{eq:def-Hpm}
h_+(z)=H\bigl(\varphi^{-1}(z)\bigr),
\qquad h_-(z)=H\Bigl(\frac{1}{\varphi^{-1}(z)}\Bigr).
\end{equation}
On the arc \(\gamma\), this agrees with the two boundary values taken from the
two sides of the arc. This has already been used throughout in section 2.

\begin{thm}
\label{thm:dual-Faber-asymp}
Assume that \(F = f\circ\varphi\) has an analytic extension 
to a neighborhood of \(\overline{D}\). 
Then, uniformly on a fixed neighborhood of \(\gamma\), we have
\[
F_{n-1}\Bigl(f\phi/\sqrt{\zeta^2-1},z\Bigr)=
\frac{1}{\sqrt{z^2-1}}\left(f_+(z)\phi_+(z)^n-f_-(z)\phi_-(z)^n\right)+O(r^n),	
\]
for some \(0<r<1\).
\end{thm}

\begin{rem}
\label{rem:removable-sing-U}
Note that the right-hand side 
\[
U(z)=\frac{1}{\sqrt{z^2-1}}\big(f_+(z)\phi_+(z)^n - f_-(z)\phi_-(z)^n\big)
\]
in this asymptotic
formula has an apparent singularity at the endpoints.
Before we turn to the proof of
Theorem~\ref{thm:dual-Faber-asymp}, we explain that 
this singularity is in fact removable
and that \(U\) is holomorphic in a neighborhood of \(\gamma\).
Indeed, the inverse Joukowski map \(\varphi^{-1}(z)\) has involutive 
boundary values on \(\gamma\), that is,  
\[
\varphi^{-1}_+(z)=\frac{1}{\varphi^{-1}_{-}(z)},
\]
and the square root map satisfies \((\sqrt{z^2-1})_+=-(\sqrt{z^2-1})_-\).
Thus, the boundary values \(U_\pm(z)\) on \(\gamma\) are given by
\[
\begin{aligned}
U_+(z)&=\frac{1}{(\sqrt{z^2-1})_+}\Big(F\big(\varphi^{-1}_+(z)\big)
\Phi^n\big(\varphi^{-1}_+(z)\big)
- F\big(1/\varphi^{-1}_+(z)\big)\Phi^n\big(1/\varphi^{-1}_+(z)\big)\Big) \\
&=-\frac{1}{(\sqrt{z^2-1})_-}\Big(F\big(1/\varphi^{-1}_-(z)\big)
\Phi^n\big(1/\varphi^{-1}_-(z)\big)
- F\big(\varphi^{-1}_-(z)\big)\Phi^n\big(\varphi^{-1}_-(z)\big)\Big)\\ 
&=U_-(z).
\end{aligned}
\]
It follows from Morera's theorem that \(U\) extends holomorphically across
the interior of the arc, so \(U\) has at most isolated singularities at \(\pm 1\). 
Since evidently \(|U(z)|=O(|z^2-1|^{-1/2})\), the apparent
singularities at \(\pm1\) are removable as well.
\end{rem}

\begin{proof}[Proof of Theorem~\ref{thm:dual-Faber-asymp}]
We will base our proof on \eqref{eq:Faber_curve_asymptotics}.
Let \(\zeta = \varphi(s)\), and \(z = \varphi(\sigma)\) 
(either of the at most two possible choices of \(\sigma\) may be used). Then
\[
\frac{d\zeta}{\sqrt{\zeta^2-1}} = \frac{ds}{s}, 
\quad \varphi(s)-\varphi(\sigma) = \frac{s^{-1}}{2}(s-\sigma)(s-\sigma^{-1})
\]
and consequently, the Faber polynomial can be written as
\begin{align*}
\frac{1}{2\pi i}\int_C\frac{f(\zeta)\phi^n(\zeta)}{\sqrt{\zeta^2-1}}
\frac{d\zeta}{\zeta-z} 
&= \frac{1}{2\pi i}\int_{\varphi^{-1}(C)}\frac{f\circ \varphi(s)\, 
(\phi\circ\varphi)^n(s)}{\varphi(s)-\varphi(\sigma)}\frac{ds}{s} \\ 
&= \frac{1}{\pi i}\int_{\varphi^{-1}(C)}\frac{f\circ\varphi(s)\, 
(\phi\circ\varphi)^n(s)}{(s-\sigma)(s-\sigma^{-1})}ds.
\end{align*}
By a partial fraction decomposition, this equals
\begin{align*}
\frac{2}{\sigma-\sigma^{-1}}\bigg(\frac{1}{2\pi i}
\int_{\varphi^{-1}(C)}\frac{F(s)\Phi^n(s)}{s-\sigma}\,ds
-\frac{1}{2\pi i}
\int_{\varphi^{-1}(C)}\frac{F(s)\Phi^n(s)}{s-\sigma^{-1}}\,ds\bigg).
\end{align*}
Applying \eqref{eq:Faber_curve_asymptotics} we find that 
for \(\sigma\) in a neighborhood of \(\partial D\), 
we have the uniform asymptotic formula
\begin{multline*}
\frac{1}{2\pi i}\int_{\varphi^{-1}(C)}
\frac{F(s)\Phi^n(s)}{s-\sigma}\,ds
-\frac{1}{2\pi i}\int_{\varphi^{-1}(C)}
\frac{F(s)\Phi^n(s)}{s-\sigma^{-1}}\,ds \\
=F(\sigma)\Phi^n(\sigma)
-F(\sigma^{-1})\Phi^n(\sigma^{-1})+O(r^n).
\end{multline*}
It is easy to see that the error term is antisymmetric under
\(\sigma\mapsto\sigma^{-1}\). It is therefore divisible by
\(\sigma-\sigma^{-1}\) and the quotient is still of order \(O(r^n)\). 
It follows that, for \(z\) in a
neighborhood of \(\gamma\),
\[
F_{n-1}\big(f\phi/\sqrt{\zeta^2-1},z\big) 
= \frac{2}{\sigma-\sigma^{-1}}\Bigl(F(\sigma)\Phi^n(\sigma)
-F(\sigma^{-1})\Phi^n(\sigma^{-1})\Bigr)+O(r^n).
\]
The claim then follows from the fact that 
\(\tfrac12(\sigma-\sigma^{-1}) = \sqrt{z^2-1}\).
\end{proof}

\subsection{A jump problem for the dual weight}
\label{s:jump-dual}
It remains to choose the analytic function \(f\), subject to \(f(\infty)=1\). 
The preceding theorem shows that the zeros of
\(\mathcal{E}_0\) are controlled by the expression
\[
f_+(z)\phi_+^n(z)-f_-(z)\phi_-^n(z).
\]
In order for these zeros to lie on the curve \(\gamma\),
or at least exponentially close, we should ensure that
\begin{equation}
\label{eq:modulus-match}
|f_+(z)|=|f_-(z)|.
\end{equation}
To make sure that the zeros occur at the extremal points of
\[
g_+(z)\phi_+^n(z)+g_-(z)\phi_-^n(z),
\]
we should arrange for the phase difference of \(f\) to match that of
\(g\). That is,
\begin{equation}
\label{eq:phase-jump}
\arg \big(f_+/f_-\big)=\arg \big(g_+/g_-\big).
\end{equation}
We introduce the auxiliary function,
\[
R(z)=\frac{\phi'(z)}{\extremal(z)},
\]
which satisfies
\[
|R_\pm(z)|=|\phi_+'(z)|+|\phi_-'(z)|,
\qquad
\frac{R(\infty)}{\phi'(\infty)}=\frac{1}{\extremal(\infty)}.
\]
Writing
\[
f(z)=\frac{R(\infty)}{R(z)}\eta(z),
\]
the desired jump relation \eqref{eq:modulus-match}--\eqref{eq:phase-jump} 
becomes
\begin{equation}
\label{eq:eta-jump}
\frac{\eta_+(z)}{\eta_-(z)}
= e^{i\arg(\phi_+'(z)/\phi_-'(z))},
\end{equation}
with normalization \(\eta(\infty)=1\).
The right-hand side can be identified geometrically. 
Indeed, the two tangents to
\(\gamma\) are given by
\[
\tau_\pm(z)=\partial_\theta\phi^{-1}(e^{i\theta})\big|_{e^{i\theta}=\phi_\pm(z)}
= i\frac{\phi_\pm(z)}{\phi_\pm'(z)}.
\]
They point in antipodal directions, and therefore
\[
\frac{\tau_+(z)/|\tau_+(z)|}{\tau_-(z)/|\tau_-(z)|}
=\frac{\phi_+(z)}{\phi_-(z)}e^{-i\arg(\phi_+'(z)/\phi_-'(z))}
=-1.
\]
Thus
\[
e^{i\arg(\phi_+'(z)/\phi_-'(z))} = -\frac{\phi_+(z)}{\phi_-(z)}.
\]
We are therefore led to solve
\[
\frac{\eta_+(z)}{\eta_-(z)}=-\frac{\phi_+(z)}{\phi_-(z)}
\]
with \(\eta(\infty)=1\), and with endpoint singularities of the type
\((z^2-1)^{-1/2}\) to compensate for the endpoint behavior coming from \(1/R\).
This gives explicitly
\[
\eta(z) = \frac{1}{\phi'(\infty)}\frac{\phi(z)}{\sqrt{z^2-1}},
\]
and hence
\begin{equation}
\label{eq:dual-weight}
f(z)=\frac{1}{\extremal(\infty)}
\frac{\phi(z)\extremal(z)}{\sqrt{z^2-1}\,\phi'(z)}
=\frac{R(\infty)}{\phi'(\infty)}
\frac{\phi(z)}{\sqrt{z^2-1}\,R(z)}.
\end{equation}

We summarize the outcome in the following lemma.

\begin{lem}
\label{lem:Faber-dual-asymp}
The function \(F\coloneqq f\circ\varphi\) extends 
to a zero-free holomorphic function on a neighborhood 
of \(\overline D\), and we have that
\[
F_{n-1}\left(f\phi/\sqrt{\zeta^2-1},z\right)
=\frac{1}{z^2-1}\frac{R(\infty)}{\phi'(\infty)}
\left(\frac{\phi_+(z)^{n+1}}{R_+(z)}+\frac{\phi_-(z)^{n+1}}{R_-(z)}\right)
+O(r^n)
\]
for some \(0<r<1\). Moreover, the extremal points \(z_0,\dotsc,z_{n}\) from
Lemma~\ref{lem:extreme-pts} are precisely the points for which
\begin{equation}
\frac{\phi_+(z)^{n+1}}{R_+(z)}
+\frac{\phi_-(z)^{n+1}}{R_-(z)}=0.
\label{eq:vanishing_cond}
\end{equation}
\end{lem}

\begin{proof}
First of all, we verify that \(f\) is an admissible weight. 
We have already seen that
\[
G(\sigma)\coloneqq g\circ\varphi(\sigma)
\]
has an analytic extension to a neighborhood of \(\overline{D}\).
Note further that
\[
h(z)\coloneqq \sqrt{z^2-1} \phi'(z)
\]
satisfies
\[
H(\sigma) = h\circ\varphi(\sigma) 
= \frac{\sigma-\sigma^{-1}}{2}\phi'\circ\varphi(\sigma) 
=\sigma \Phi'(\sigma)
\]
which clearly has an analytic extension and doesn't 
vanish in a neighborhood of \(\overline{D}\).
Since
\begin{equation}
\label{eq:fgh-rel}
f(z) = \phi(z)g(z)/h(z)
\end{equation}
we find that 
\[
F = f\circ\varphi = \Phi(\sigma)G(\sigma)/H(\sigma)
\] 
has an analytic extension to a neighborhood of \(\overline{D}\).
The corresponding Faber polynomial satisfies
\begin{align*}
F_{n-1}(\phi f/\sqrt{\zeta^2-1},z) &
= \frac{1}{\sqrt{z^2-1}}\Bigl(f_+(z)\phi_+(z)^{n}-f_-(z)
\phi_-(z)^{n}\Bigr)+O(r^n) \\
& = \frac{1}{z^2-1}\frac{R(\infty)}{\phi'(\infty)}
\Bigl(\frac{\phi_+(z)^{n+1}}{R_+(z)}+\frac{\phi_-(z)^{n+1}}{R_-(z)}\Bigr)+O(r^n).
\end{align*}
This proves the claimed asymptotic formula.

The conclusion concerning the extreme points 
follows directly from how we set up the jump problem: 
by the phase condition \eqref{eq:phase-jump}, the maxima of
\(\lvert g_+\phi_+^n+g_-\phi_-^n\rvert\) occur precisely at the minima of
\(\lvert f_+\phi_+^n-f_-\phi_-^n\rvert\), while the
matching condition \eqref{eq:modulus-match} for the moduli \(|f_\pm|\)
forces the minima to be zero. 
Substitution of \eqref{eq:dual-weight} gives
precisely \eqref{eq:vanishing_cond}.
\end{proof}

\subsection{The derivative \texorpdfstring{\(\mathcal E'(z)\)}{} and the lower bound}
\label{ss:derivative}
We keep the function \(f\) from \eqref{eq:dual-weight}, 
and introduce the polynomial
\begin{equation}
\label{eq:E0-def}
\mathcal{E}_0(z)=(z^2-1)
F_{n-1}\big(f\phi/\sqrt{\zeta^2-1},z\big),
\end{equation}
which should be a good approximation of the
polynomial \(\mathcal E_{n+1}\) with zeros at the
extreme points of \(F_n(g,z)\) along the arc \(\gamma\).

Let \(\Phi=\phi\circ\varphi\), and note that
\[
f\circ\varphi(\sigma)
=\frac{R(\infty)}{\phi'(\infty)}
\frac{\Phi(\sigma)}
{(\tfrac{\sigma-\sigma^{-1}}{2})R(\varphi(\sigma))}.
\]
We also define the auxiliary function
\[
\mathcal R(\sigma)\coloneqq
\Bigl(\frac{\sigma-\sigma^{-1}}{2}\Bigr)R(\varphi(\sigma)).
\]
Since \(\gamma\) is assumed to be analytic, \(\Phi\) is analytic and nonzero
in a neighborhood of \(\partial D\), and the same holds for
\(1/\mathcal R\), cf.\ the proof of Lemma~\ref{lem:Faber-dual-asymp}.
We can therefore define the two functions
\begin{equation}
W(\sigma)\coloneqq
\frac{R(\infty)}{\phi'(\infty)}
\left(\frac{\Phi(\sigma)^{n+1}}{\mathcal R(\sigma)}
-\frac{\Phi(\sigma^{-1})^{n+1}}{\mathcal R(\sigma^{-1})}\right),
\label{eq:def_W}
\end{equation}
and
\begin{equation}
U(z)\coloneqq\sqrt{z^2-1}\,W(\varphi^{-1}(z))
=\frac{R(\infty)}{\phi'(\infty)}
\left(\frac{\phi_+(z)^{n+1}}{R_+(z)}
+\frac{\phi_-(z)^{n+1}}{R_-(z)}\right).
\label{eq_def_U}
\end{equation}
By Remark~\ref{rem:removable-sing-U}, \(U\) is holomorphic in a
neighborhood \(\mathcal N\) of \(\gamma\), and, after possibly adjusting this
neighborhood, Lemma~\ref{lem:Faber-dual-asymp} implies that
\begin{equation}
\label{eq:E0-U-est}
|\mathcal E_0(z)-U(z)|=O\bigl(|z^2-1|r^n\bigr)
\end{equation}
holds for some \(0<r<1\), uniformly on \(\mathcal N\).
By a Cauchy estimate, we therefore have
\[
\mathcal E_0'(z)=U'(z)+O(r^n)
\]
uniformly on a slightly smaller neighborhood.
Moreover, the zeros of \(U\) coincide exactly with the \(n+1\)
extreme points \(z_j=z_{j,n}\) of
\[
g_+(z)\phi_+(z)^n+g_-(z)\phi_-(z)^n
\]
along \(\gamma\).

To calculate the derivative of \(U\) at \(z_j\), note first that
\[
\frac{d}{dz}\varphi^{-1}(z)
=\frac{\varphi^{-1}(z)}{\sqrt{z^2-1}},
\]
and therefore
\begin{equation}
\label{eq:U-prim-decomp}
U'(z)=\frac{z}{\sqrt{z^2-1}}W(\varphi^{-1}(z))
+W'(\varphi^{-1}(z))\varphi^{-1}(z).
\end{equation}
Differentiating \(W\) directly yields
\begin{align}
W'(\sigma)
&=(n+1)\frac{R(\infty)}{\phi'(\infty)}
\left(\frac{\Phi'(\sigma)\Phi(\sigma)^n}{\mathcal R(\sigma)}
+\frac{\sigma^{-2}\Phi'(\sigma^{-1})
\Phi(\sigma^{-1})^n}{\mathcal R(\sigma^{-1})}\right)
\\
&\quad\, 
-\frac{R(\infty)}{\phi'(\infty)}
\left(\frac{\Phi(\sigma)^{n+1}\mathcal R'(\sigma)}{\mathcal R(\sigma)^2}
+\frac{\sigma^{-2}\Phi(\sigma^{-1})^{n+1}
\mathcal R'(\sigma^{-1})}{\mathcal R(\sigma^{-1})^2}\right).
\end{align}
Note in particular that, since \(|\Phi(\sigma)|=1\)
for \(\sigma\in\partial D\) and \(1/\mathcal R\) is
analytic, the last term is bounded by a constant independent of \(n\)
on \(\partial D\). We find that
\begin{equation}
\label{eq:W-prim-asymp}
\begin{aligned}
W'(\varphi^{-1}(z))\varphi^{-1}(z)
&=(n+1)\frac{R(\infty)}{\phi'(\infty)}
\left(\frac{\phi_+'(z)\phi_+(z)^n}{R_+(z)}
+\frac{\phi_-'(z)\phi_-(z)^n}{R_-(z)}\right)+O(1)\\ 
&= (n+1)\big(g_+(z)\phi_+^n(z)+g_-(z)\phi_-^n(z)\big)+O(1),
\end{aligned}
\end{equation}
where the \(O(1)\)-term is uniformly bounded for \(z\in\gamma\). 
In particular,
\[
\left|W'(\varphi^{-1}(z_j))\varphi^{-1}(z_j)\right|
=
\frac{n+1}{\extremal(\infty)}+O(1),\qquad j=0,1,\dotsc,n.
\]
The remaining term in \eqref{eq:U-prim-decomp},
\[
\frac{z}{\sqrt{z^2-1}}W(\varphi^{-1}(z)),
\]
vanishes at all the points \(z_j\neq\pm1\). Since additionally
\(W(\pm1)=0\), we find that
\begin{equation}
\label{eq:deriv-term-end-pt}
\begin{aligned}
\lim_{z\to\pm1}
\frac{z}{\sqrt{z^2-1}}W(\varphi^{-1}(z))
&=\lim_{\sigma\to\pm1}
\frac{\varphi(\sigma)}{\tfrac{\sigma-\sigma^{-1}}{2}}W(\sigma)
\\
&=\lim_{\sigma\to\pm1}
\frac{2\sigma\varphi(\sigma)}{(\sigma-1)(\sigma+1)}W(\sigma)
\\
&=\pm W'(\pm1)=\lim_{z\to\pm1}W'(\varphi^{-1}(z))\varphi^{-1}(z).
\end{aligned}
\end{equation}
Consequently, we may write
\[
|U'(\pm1)|=\left|2\lim_{z\to\pm1}W'(\varphi^{-1}(z))\varphi^{-1}(z)\right|
=\frac{2(n+1)}{\extremal(\infty)}+O(1).
\]
Combining these findings, we conclude that there exists \(C>0\) such that
\begin{equation}
\label{eq:ineq-E0}
|\mathcal E_0'(z_j)| \geq \frac{n+1}{\extremal(\infty)}-C,
\qquad j=0,\dotsc,n.
\end{equation}

The polynomial \(\mathcal E_0\) has its
zeros very close to, but not necessarily exactly
at, the extremal points \(z_j\). In fact, the zeros may lie slightly 
away from the arc \(\gamma\).
We let \(\mathcal E=\mathcal E_{n+1}\) denote the polynomial whose
zeros are exactly the points \(\{z_j\}_{j=0}^n\) with the same
leading coefficient as \(\mathcal E_0\).

\begin{lem}
\label{lem:E-prim}
There exists \(c>0\) such that
\[
\mathcal E'(z_j)=\mathcal E_0'(z_j)\bigl(1+O(e^{-cn})\bigr)
\]
uniformly for \(0\leq j\leq n\). In particular, there exists a constant
\(C\) depending only on \(\gamma\) such that
\[
|\mathcal E'(z_j)|\geq\frac{n+1}{\extremal(\infty)}-C,
\qquad 0\leq j\leq n.
\]
\end{lem}

\begin{proof}
Let us first quantify how close the zeros of
\(\mathcal E_0\) are to \(z_j\).
The endpoints \(z_0=1\) and \(z_n=-1\) are already joint zeros
of the two polynomials, so it suffices to consider
\(j=1,\dotsc,n-1\).
Since
\[
\max\{|\phi_+(z)|,|\phi_-(z)|\}\leq1+O(n^{-1})
\]
whenever \(\operatorname{dist}(z,\gamma)\leq n^{-2}\),
\eqref{eq_def_U} shows that \(U(z)\) is uniformly bounded within a
distance \(O(n^{-2})\) of the arc \(\gamma\).
From standard Cauchy estimates it follows that, for some \(\delta>0\),
\begin{equation}
\label{eq:bound-second-deriv}
\max_{|z-z_j|\leq\delta n^{-2}}|U''(z)|=O(n^4).
\end{equation}
Thus, if \(z\) is sufficiently close to \(z_j\), Taylor's formula gives that
\[
U(z)=U'(z_j)(z-z_j)+O\bigl(n^4|z-z_j|^2\bigr)
\]
and hence with 
\[
|z-z_j|=C_0 r^n n^{-1}
\]
we obtain
\[
|U(z)|\geq |U'(z_j)|C_0 r^n n^{-1}-O\bigl(n^2r^{2n}\bigr).
\]
The second term is exponentially smaller than the first. Since
\eqref{eq:ineq-E0} and the estimate \(\mathcal E_0'-U'=O(r^n)\) 
give \(|U'(z_j)|\gtrsim n\), while the estimate \eqref{eq:E0-U-est} shows that
\(|U(z)-\mathcal E_0(z)|=O(r^n)\), choosing \(C_0\) sufficiently
large gives
\[
|U(z)|>|U(z)-\mathcal E_0(z)|
\]
on the circle \(|z-z_j|=C_0 r^n n^{-1}\). Rouch\'e's theorem therefore shows that
\(\mathcal E_0\) has a zero \(w_j\) satisfying
\[
|z_j-w_j|=O(r^n n^{-1}).
\]

We next perturb the zeros of \(\mathcal E_0\) to the
actual locations \(z_j\in\gamma\), and need to ensure that this does not
ruin our estimates. From Lemma~\ref{lem:extreme-pts}, we have
\begin{equation}
\label{eq:dist-zj-zk}
|z_j-z_k|\geq\frac{c}{n^2},
\qquad j\neq k.
\end{equation}
Thus, for \(n\) sufficiently large, the discs above are disjoint.
Together with the exact zeros at the endpoints, they account for all
\(n+1\) zeros of \(\mathcal E_0\), and hence
\begin{equation}
\mathcal E_0(z)=\phi'(\infty)^n\prod_{j=0}^n(z-w_j).
\end{equation}
The estimate \eqref{eq:E0-U-est}, together with
\eqref{eq:bound-second-deriv}, gives
\(\mathcal E_0''=O(n^4)\) between \(z_j\) and \(w_j\).
Thus,
\[
\mathcal E_0'(z_j)=\mathcal E_0'(w_j)
+\int_{w_j}^{z_j}\mathcal E_0''(s)\,ds
=\mathcal E_0'(w_j)+O(n^3r^n),
\]
and consequently
\begin{equation}
\label{eq:E0-prim-wj}
\mathcal E_0'(w_j)=\mathcal E_0'(z_j)\bigl(1+O(e^{-cn})\bigr).
\end{equation}
Next,
\begin{align*}
\frac{\mathcal E_0'(w_j)}{\mathcal E'(z_j)}
=\prod_{\substack{k=0\\k\neq j}}^n\frac{w_j-w_k}{z_j-z_k}
&=\prod_{\substack{k=0\\k\neq j}}^n
\Big(1+\frac{(w_j-z_j)+(z_k-w_k)}{z_j-z_k}\Big) \\
&=\prod_{\substack{k=0\\k\neq j}}^n
\Big(1+\frac{O(r^n/n)}{z_j-z_k}\Big),
\end{align*}
and the separation \eqref{eq:dist-zj-zk} of zeros therefore gives
\[
\left|\frac{\mathcal E_0'(w_j)}{\mathcal E'(z_j)}-1\right|
\leq Cn^2r^n=O(e^{-cn}).
\]
Combining this with \eqref{eq:E0-prim-wj} and \eqref{eq:ineq-E0} proves the lemma.
\end{proof}

\subsection{Asymptotics of Widom factors}
We are now ready to prove the first of our two main results.

\begin{proof}[Proof of Theorem~\ref{thm:main}]
The upper bound
\[
\limsup_{n\to \infty} \mathcal{W}_n(\gamma)\le 1/\extremal(\infty)
\]
was established above in Theorem~\ref{thm:optimal_faber}.
To obtain the matching lower bound we first appeal to Theorem~\ref{thm:opm}, 
or more specifically to its reformulation \eqref{eq:cor-opm},
which established that
\[
\mathcal{W}_n(\gamma)\ge \Big(\sum_{j=0}^n \frac{1}{|\mathcal E'(z_j)|}\Big)^{-1},
\]
where \(\{z_j\}_{j=0}^n\) are any \(n+1\) distinct points on \(\gamma\) 
and where 
\[
\mathcal E(z)=\phi'(\infty)^n\prod_{j=0}^n (z-z_j).
\]
As above, we pick the points \(\{z_j\}_{j=0}^n\) as the extreme points 
from Lemma~\ref{lem:extreme-pts}. 
With this choice of \(\mathcal E\), Lemma~\ref{lem:E-prim} gives the upper bound
\[
\frac{1}{|\mathcal E'(z_j)|}\le \frac{\extremal(\infty)}{n+1} \big(1+O(1/n)\big)
\] 
so that
\begin{align}
\mathcal{W}_n(\gamma)\ge \Big(\sum_{j=0}^n \frac{1}{|\mathcal E'(z_j)|}\Big)^{-1} 
&\ge \Big(\frac{1}{n+1}\sum_{j=0}^n\extremal(\infty)\Big)^{-1}\big(1+O(1/n)\big) \\
& = 1/\extremal(\infty) + O(1/n).
\end{align}
Taking the lower limit as \(n\to\infty\) establishes the lower bound
\[
\liminf_{n\to\infty} \mathcal{W}_n(\gamma)\ge 1/\extremal(\infty)
\]
for the Widom factors, and hence the proof is complete.
\end{proof}

\section{Szeg\H{o}-Widom asymptotics}
% !TeX root = ../main.tex

\label{s:asymp}

\subsection{Stability of optimal prediction polynomials}
We now turn to our second main result, Theorem~\ref{thm:SzegoWidom}. 
We begin with a stability lemma for {\em optimal prediction polynomials}, that is,
the orthogonal polynomials associated with optimal prediction measures.

To formulate the result, we denote by \(P_n\) 
a sequence of monic polynomials of degree \(n\), and write
\[
M_n=\lVert P_n\rVert_\gamma.
\]
Suppose that \(z_0,z_1,\ldots,z_n\) are any \(n+1\) distinct
points on \(\gamma\) and that
\(\lambda_{0},\dots,\lambda_{n}\) are strictly positive weights summing to 1, 
and let \(\mu_n\) be the probability measure
\[
\mu_n=\sum_{0\le j\le n} \lambda_{j} \delta_{z_{j}}.
\]
It turns out that an approximate orthogonality condition for 
\(P_n\) with respect to \(\mu_n\) together with an almost 
extremality condition on \(\operatorname{supp}\mu_n\) imply that
\(P_n\) is close to the Chebyshev polynomial \(T_n\). 
Specifically, we assume that:

\begin{itemize}[leftmargin=.8cm] 
\item[{\rm (i)}] \(P_n\) is close to extremal on the \(z_j\). That is,
for some positive sequence \(\epsilon_n\to 0\),
\[
|P_n(z_{j})| \ge M_n-\epsilon_n,
\qquad 0\le j\le n.
\]

\item[{\rm (ii)}] \(P_n\) is almost orthogonal with respect to \(\mu_n\).
Specifically, for some positive sequence \(\delta_n\to 0\), we have
\[
\Big|\int_{\gamma} P_n(z)\overline{q(z)}d\mu_n(z)\Big|
\le \delta_n\, \lVert q\rVert_{L^2(\mu_n)}
\]
for any polynomial \(q\) of degree at most \(n-1\).
\end{itemize}

\medskip

\begin{lem}
\label{lem:stab}
With the above notation and conditions {\rm (i)-(ii)} satisfied, we have
\[
\lVert T_n-P_n\rVert_{L^2(\mu_n)}
\le \delta_n  + \sqrt{\delta_n^2 + 2M_n\epsilon_n}.
\]
\end{lem}

\begin{proof}
Set \(D_n:=T_n-P_n\) and note that
\begin{equation}
\label{eq:main-identity}
\lVert T_n\rVert^2_{L^2(\mu_n)}=\lVert P_n\rVert^2_{L^2(\mu_n)}
+\lVert D_n \rVert^2_{L^2(\mu_n)}
+2\Re \big\langle D_n,P_n\big\rangle_{L^2(\mu_n)}
\end{equation}
Since \(T_n\) is the monic Chebyshev polynomial 
and since \(P_n\) is monic of degree \(n\),
\[
\lVert T_n\rVert_{L^2(\mu_n)}^2\le 
\|T_n\|_\gamma^2 \le \|P_n\|_\gamma^2 = M_n^2.
\]
By rearranging \eqref{eq:main-identity}, we thus get
\[
\lVert D_n\rVert_{L^2(\mu_n)}^2
\le M_n^2 -\lVert P_n\rVert_{L^2(\mu_n)}^2
- 2\Re\, \langle D_n,P_n\rangle_{L^2(\mu_n)}.
\]
By Assumption {\rm (i)}, \(P_n\) almost reaches 
its extremal modulus \(M_n\) on each of the mass 
points of \(\mu_n\), so
\[
|P_n(z_j)| \ge M_n-\epsilon_n,
\]
which gives the bound from below for the \(L^2\)-norm,
\[
\lVert P_n\rVert_{L^2(\mu_n)}^2
=\sum_{j=0}^n \lambda_j |P_n(z_j)|^2 
\ge (M_n-\epsilon_n)^2.
\]
Since both \(T_n\) and \(P_n\) are monic of degree \(n\), 
the difference \(D_n\) has degree at most \(n-1\).
Thus we may apply assumption {\rm (ii)} to \(q=D_n\), and get
\[
\Re\,\langle D_n,P_n\rangle_{L^2(\mu_n)}\le 
\big|\langle D_n,P_n\rangle_{L^2(\mu_n)}\big|
\le \delta_n \|D_n\|_{L^2(\mu_n)},
\]
leading to the upper bound
\[
\lVert D_n\rVert_{L^2(\mu_n)}^2
\le M_n^2-(M_n-\epsilon_n)^2 + 2\,\delta_n\|D_n\|_{L^2(\mu_n)}
\]
for the norm of the difference \(D_n\).
Rearranging and completing the square, we arrive at
\[
\begin{aligned}
\left(\lVert D_n\rVert_{L^2(\mu_n)} - \delta_n \right)^2 &
\le M_n^2-(M_n-\epsilon_n)^2 + \delta_n ^2 \\
& = 2M_n\epsilon_n-\epsilon_n^2+ \delta_n ^2\\ 
&\le 2M_n\epsilon_n + \delta_n ^2.
\end{aligned}
\]
This gives that
\[
\lVert D_n\rVert_{L^2(\mu_n)}\le \delta_n  
+ \sqrt{2M_n\epsilon_n + \delta_n^2},
\]
completing the proof.
\end{proof}

Note that if \(\epsilon_n\) and \(\delta_n\) 
both vanish, then we have \(P_n=T_n\).

\subsection{A Marcinkiewicz--Zygmund inequality}
To go from stability in the discrete \(L^2(\mu_n)\)-norm 
to stability in \(L^2(\gamma,|dz|)\), we need a sampling inequality. 
The following result is distilled from \cite{ChuiZhong}*{Section~2}.
For the formulation, we need the notion of the Muckenhoupt 
{\em \(A_2\)-characteristic} of a weight function \(\omega(z)\):
Given a contour \(\Gamma\), we let

\begin{equation}\label{eq:def-A2-constant}
[\omega]_{A_2,\Gamma}=\sup_{I\subset \Gamma}
\Big(\frac{1}{|I|}\int_{I} \omega(z)\,|dz|\Big)
\Big(\frac{1}{|I|}\int_{I} \omega^{-1}(z)\,|dz|\Big),
\end{equation}
where the supremum is taken over all sub-arcs of \(\Gamma\).
It is well known that many singular integral operators are bounded on weighted
\(L^2\)-spaces precisely when \([\omega]_{A_2,\Gamma}\) is finite. 
In the proof of the below lemma, 
it enters as an upper bound for the norm
of the Cauchy transform on \(L^2(\gamma_n,\omega\,|dz|)\), 
where \(\gamma_n\) denotes the Green curve
\[
\gamma_n=\Big\{z\in \C: |\phi(z)|=1+\tfrac1n\Big\}.
\]

\begin{lem}
\label{lem:MZ}
Let \(z_0,\ldots,z_n\) be distinct points on \(\gamma\), 
ordered by increasing arc-length from one of the end-points,
such that
\begin{equation}
\label{eq:unif-sep}
|z_j-z_{j-1}|\ge c_0 \min\{d(z_j,\gamma_n), d(z_{j-1},\gamma_n)\}.
\end{equation}
Denote by \(\mathcal E\) the polynomial 
\(\mathcal E(z)=\phi'(\infty)^{n}\prod_{j=0}^n (z-z_j)\) and put
\(\omega_{\mathcal E}(z)=|\mathcal E(z)|^2\).
Then there exists a constant \(C\) depending only on \(c_0\) and \(\gamma\)
such that, for any polynomial \(P\) of degree \(n\),
\[
\int_{\gamma}|P(z)|^2\,|dz|
\le C\frac{[\omega_{\mathcal E}]_{A_2,\gamma_n}^2}{n}\sum_{j=0}^n |P(z_j)|^2.
\]
\end{lem}

Lemma~\ref{lem:MZ} is taken from \cite{ChuiZhong}*{Section~2}.
The authors of that paper are mainly interested in finer sampling theory
for one particular (\enquote{\(L\)-shaped}) arc, and therefore do not
state their results in this generality. However, the proof works 
essentially verbatim for analytic arcs. In Appendix~\ref{s:MZ} below, 
we give a detailed proof sketch and a guide to the necessary modifications.

If we fix \(z_0,\ldots,z_n\) to be the extreme points of
\[
\big|g_+(z)\phi_+(z)^n + g_-(z)\phi_-(z)^n\big|
\]
along \(\gamma\) and if \(\mathcal E\) is the associated dual polynomial
normalized by capacity, then an elementary computation reveals that
the Muckenhoupt characteristics satisfy
\begin{equation}
\label{eq:A2-constant-E}
[\omega_{\mathcal E}]_{A_2,\gamma_n}\asymp \log n.
\end{equation} 
The uniform separation \eqref{eq:unif-sep}
also holds, and in fact we also have a corresponding bound from above.
Hence, for that particular choice of points, Lemma~\ref{lem:MZ}
together with the bound \eqref{eq:A2-constant-E}
gives
\begin{equation}
\label{eq:MZ-ineq}
\int_{\gamma}|P(z)|^2\,|dz|\le \frac{C(\log n)^2}{n}\sum_{j=0}^n |P(z_j)|^2.
\end{equation}
The verification of the bound \eqref{eq:A2-constant-E} on the Muckenhoupt characteristic 
and of the uniform separation \eqref{eq:unif-sep} are both elementary but quite lengthy, 
and for this reason we postpone their proofs to Appendix~\ref{s:geom-lemmas}, 
see Proposition~\ref{prop:A2-constant-E} and Lemma~\ref{lem:separation}. 

\subsection{Approximate orthogonality of \texorpdfstring{\(F_n(g,z)\)}{}}
The final step towards the proof of Theorem~\ref{thm:SzegoWidom}
is the verification of the conditions of Lemma~\ref{lem:stab} for
our trial polynomial \(F_n(g,z)\).
We let \(z_j\) and \(\mathcal E(z)\) be as above, and 
let \(\mu_n\) be the measure
\[
\mu_n=\sum_{j=0}^n \frac{1}{|\mathcal E'(z_j)|}\delta_{z_j}
\Big/ \sum_{j=0}^n \frac{1}{|\mathcal E'(z_j)|}.
\]

\begin{prop}
\label{prop:approx-orthogonality}
The weighted Faber polynomial \(F_n(g,z)\) satisfies
\[
|F_n(g,z_j)|=\lVert F_n(g,\cdot) \rVert_{\gamma} 
+ O(e^{-cn}),\qquad 0\le j\le n,
\]
as well as the approximate orthogonality for \(\deg(q)<n\)
\[
\Big|\int F_n(g,z)\overline{q(z)}d\mu_n(z)\Big|
\le\frac{C}{n}\lVert q\rVert_{L^2(\mu_n)}.
\]
\end{prop}

\begin{proof}
Let \(F_n=F_n(g,z)\) be the near-optimal Faber polynomial
\[
F_n(g,z)=\frac{1}{2\pi i}\int_{C}\frac{g(w)\phi^n(w)}{w-z}dw,
\qquad g(z) =\extremal(z)/\extremal(\infty).
\]
Since \(|g_+|+|g_-|\) is constant along \(\gamma\), 
the asymptotics \eqref{eq:faber_pol_abs} together with 
the defining property of the extreme points \(\{z_j\}_{j=0}^n\) 
from Lemma~\ref{lem:extreme-pts} imply that
\begin{align*}
|F_n(g,z_{j})| = \|F_n\|_\gamma + O(e^{-cn}),
\end{align*}
which verifies the first assertion.

It thus remains to verify the approximate orthogonality of \(F_n(g,z)\).
To this end, recall the definitions \eqref{eq:E0-def}, \eqref{eq_def_U} 
of \(\mathcal E_0\) and \(U\),
respectively. Lemma~\ref{lem:E-prim} implies that
\[
\mathcal E'(z_j)=\mathcal E_0'(z_j)\big(1+O\big(e^{-cn}\big)\big)
\]
for some \(c>0\), and after possibly adjusting \(c\), we have that
\[
\mathcal E'(z_j)=U'(z_j)\big(1+O\big(e^{-cn}\big)\big).
\]
The asymptotics \eqref{eq:Faber_asymptotics_arc} of \(F_n(g,z)\) 
applied with the above choice of \(g\) combined with
the asymptotics of \(U'\) given by \eqref{eq:U-prim-decomp} and 
\eqref{eq:W-prim-asymp} gives that
\[
F_n(g,z_{j}) = \frac{U'(z_{j})}{n+1}+ O\Big(\frac{1}{n}\Big)
= \frac{\mathcal E'(z_{j})}{n+1} + O\Big(\frac{1}{n}\Big)
\]
for all interior extreme points \(z_j\). 
Furthermore, we can compute
\begin{align}
\left|\frac{1}{\extremal(\infty)}\frac{\mathcal E'(z_{j})}{|\mathcal E'(z_{j})|} 
- \frac{\mathcal E'(z_{j})}{n+1}\right| 
= \frac{1}{n+1} \left| |\mathcal E'(z_{j})| 
- \frac{n+1}{\extremal(\infty)}\right| = O\Big(\frac{1}{n}\Big).
\end{align}
which implies that
\begin{align}
F_n(g,z_{j})=\frac{1}{\extremal(\infty)}
\frac{\mathcal E'(z_{j})}{|\mathcal E'(z_{j})|} 
+ O\Big(\frac{1}{n}\Big)
\end{align}
for all the extrema \(z_j\) except the end-points. 
By Theorem~\ref{thm:opm}, it follows that 
$\mathcal E'(z_{j})/|\mathcal E'(z_{j})|$ are the values of the $n$th 
orthogonal polynomial w.r.t. the measure
\[
\mu_n =\sum_{j=0}^n\frac{1}{|\mathcal E'(z_{j})|} \delta_{z_{j}}
\bigg/\sum_{j=0}^n\frac{1}{|\mathcal E'(z_{j})|}.
\]

The end-points are treated similarly, but the asymptotics 
of \(\mathcal E_0'\) is multiplied by a factor 2 there, 
cf.\ \eqref{eq:U-prim-decomp}--\eqref{eq:deriv-term-end-pt}.
This leaves the value \(\mathcal E'(z_{j})/|\mathcal E'(z_{j})|\)
unchanged to leading order, so we get that also
\[
F_n(g,z_j)=\frac{1}{\extremal(\infty)}
\frac{\mathcal E'(z_{j})}{|\mathcal E'(z_{j})|}
+O\Big(\frac1n\Big),\qquad j=0,\ldots,n.
\]
This shows that, for any polynomial \(q\) of degree at most \(n-1\), we have
\begin{align}
\bigg|\int_{\gamma} F_n(g,z) \overline{q(z)}\,d\mu_n(z)\bigg| 
= O\Big(\frac{1}{n}\lVert q\rVert_{L^2(\mu_n)}\Big)
\end{align}
which completes the proof.
\end{proof}

\subsection{Szeg\H{o}-Widom asymptotics of Chebyshev polynomials}
We are ready for the proof of our second and final main result, which
gives strong \(L^2\) and Szeg\H{o}--Widom asymptotics of 
the Chebyshev polynomials on an analytic arc.

\begin{proof}[Proof of Theorem~\ref{thm:SzegoWidom}]
Without loss of generality, we may assume that \(\capacity(\gamma)=1\).
By Proposition~\ref{prop:approx-orthogonality} and Lemma~\ref{lem:stab}, 
the Faber polynomial \(F_n(g,z)\) satisfies
\[
\big\lVert T_n- F_n(g,\cdot)\big\rVert_{L^2(\mu_n)}\le \frac{C}{n}.
\]
We can further control the standard \(L^2\)-norm of the difference using 
the Marcinkiewicz--Zygmund inequality \eqref{eq:MZ-ineq}, 
\[
\begin{aligned}
\big\lVert T_n-F_n(g,\cdot)\big\rVert_{L^2(\gamma,|dz|)}& \le 
\Big(\frac{C(\log n)^2}{n}\sum_{j=0}^{n}|T_n(z_j)-F_n(g,z_j)|^2\Big)^{1/2} \\
& \le  C'\log  n\, \lVert T_n-F_n(g,\cdot)\rVert_{L^2(\mu_n)}\\
&\le C''\log n/n.
\end{aligned}
\]
This completes the proof of the first part of the theorem.

To get from \(L^2\) to Szeg\H{o}--Widom asymptotics, 
note that by a version of the Bernstein-Walsh lemma\footnote{This statement is given
for the interval \([-1,1]\) and an exponential weight by 
Levin-Lubinsky \cite{LevinLubinsky}*{Section 5}, but the proof 
adapts verbatim to the present setting.},
\[
|P(z)|^2\le \frac{|\phi(z)|}{d(z,\gamma)}
|\phi(z)|^{2n}\lVert P\rVert_{L^2(\gamma,|dz|)}^2.
\]
whenever \(\deg(P)\le n\). Hence, when \(z\) is 
bounded away from \(\gamma\) we get that
\[
\frac{|T_n(z)-F_n(g,z)|}{|\phi(z)|^{n}} \le \frac{C \log n}{n}.
\]
This means that
\[
\begin{aligned}
T_n(z)&=F_n(g,z) + O\Big(\tfrac{\log n}{n}|\phi(z)|^{n}\Big)\\
&=F_n(g,z)\Big(1+O\big(\tfrac{\log n}{n}\Big)\Big),
\end{aligned}
\]
which completes the proof.
\end{proof}

\appendix
\section{Marcinkiewicz--Zygmund inequalities}
% !TeX root = ../main.tex

\label{s:MZ}
In this section, we give an overview of a result of Chui--Zhong, quoted above as
Lemma~\ref{lem:MZ}. Their original result is formulated for one particular
arc, while we need it for a general analytic arc. 
We also need to allow the \(A_2\)-characteristic of the weight 
\(\omega_n(z)=|\mathcal E_{n+1}(z)|^2\) 
to degenerate with \(n\), and thus  some extra care must be taken. 
We therefore reproduce the key steps of the proof here,
with the goal to explain how the proof works and 
to point out exactly which steps are affected 
by the lack of uniformity and which are not.

We allow ourselves to use the notation \(\lesssim\), \(\gtrsim\), \(\asymp\),
etc., when it is clear that the implied constant does not depend
on \(n\).
 
\subsection{Control on Green curves}
We fix the exterior conformal map \(\phi\) and denote by \(\gamma_n\) the
Green curves
\[
\gamma_n=\big\{z\in\C: |\phi(z)|=1+1/n\big\}.
\]
The key assumption about the distribution of our chosen points \(z_j\) is that
\begin{equation}
\label{eq:separation-assumption-below}
|z_j-z_{j-1}|\gtrsim d(z_j,\gamma_n),\qquad j=1,\ldots,n,
\end{equation}
The following useful comparison result is taken from \cite{Dynkin}*{Lemma 10}.
\begin{lem}
\label{lem:compare-norms-Green}
There exist positive constants \(c\) and \(C\) such that
\[
c \int_{\gamma_n}|P(z)|^2\, |dz|\le 
\int_{\gamma}|P(z)|^2\, |dz|\le 
C\int_{\gamma_n}|P(z)|^2\, |dz|
\]
for any polynomial \(P\) of degree \(n\).
\end{lem}

\subsection{A sampling theorem for Hardy spaces}
We denote by \(D_n\) the domain enclosed by \(\gamma_n\).
One key input is the following interpolation lemma for the Hardy space
\(H^2(D_n)\). 

\begin{lem}[\cite{ChuiZhong}*{Lemma~2.3}]
\label{lem:interpolation}
There exists a constant \(C\) such that for any sequence \(\{a_{j}\}_{j=0}^n\), 
there exists \(f\in H^2(D_n)\) with \(f(z_j)=a_j\) and
\[
\int_{\gamma_n}|f(z)|^2\, |dz|\le C\sum_{j=0}^n|a_j|^2 d(z_j,\gamma_n).
\]
\end{lem}

The proof appears in \cite{ChuiZhong}*{Lemma~2.3}, although
some details are left to the reader. If \(w_{j,n}\) denote the
images of the points \(z_j\) under a conformal map from \(D_n\)
onto the unit disk, the required input for interpolation is that these
are uniformly separated, in the sense that
\[
\inf_{k=0,\dotsc,n}\prod_{j\neq k}
\left|\frac{w_{j,n}-w_{k,n}}{1-\overline{w_{k,n}}w_{j,n}}\right|
\geq c>0,
\]
with \(c\) independent of \(n\). By Carleson's interpolation
theorem, formulated for the upper half-plane in
\cite{Gar07}*{Theorem~1.1}, this is equivalent to weak
separation in the pseudohyperbolic metric together with the Carleson measure
estimate
\[
\int_{\D}|h|^2\,d\nu_n
\leq C\|h\|_{H^2(\D)}^2,
\qquad \nu_n=\sum_{j=0}^n(1-|w_{j,n}|^2)\delta_{w_{j,n}}.
\]
The required estimates, with all relevant constants uniform in \(n\), are
established in \cite{Zho94}*{Lemma~1} and
\cite{ChuiZhong-JAT}*{Lemma~5.2}. These estimates 
crucially rely on the separation assumption
\eqref{eq:separation-assumption-below}. The stated interpolation 
estimate follows by applying the \(H^2\)-interpolation
theorem of Shapiro--Shields \cite{SS61}*{Lemma~3}, followed by 
composition with an appropriate conformal map.

It is worth noting that in the proof of \cite{ChuiZhong-JAT}*{Lemma 5.2}, 
the authors already use uniform bounds for the (unweighted) Cauchy 
transform which is the topic of the next subsection.

\subsection{Uniform control of Cauchy transforms}
For a closed contour \(\Gamma\), we denote by \(\mathcal{C}=\mathcal{C}_\Gamma\) 
the Cauchy transform on \(\Gamma\), defined as
\[
\mathcal{C}f(z)={\rm p.v.}\frac{1}{2\pi i}\int_{\Gamma}\frac{f(w)}{w-z}dw.
\]
For the purposes of this discussion, we consider only smooth contours
and smooth functions \(f\), so the meaning of this is clear.
The principal value integral \(\mathcal{C}f\) is related to the boundary 
values of the standard Cauchy integral via the Sokhotskii--Plemelj formula
\begin{equation}
\label{eq:Plemelj}
\mathcal{C}f(\zeta)
=\lim_{z\to \zeta}\,\frac{1}{2\pi i}\,\int_{\Gamma}\frac{f(w)}{w-z}dw
-\frac{1}{2}f(\zeta),\qquad \zeta\in \Gamma,
\end{equation}
where the limit is taken non-tangentially from within the
region enclosed by \(\Gamma\).

We say that \(\mathcal{C}\) defines a bounded operator on 
\(L^2(\Gamma, \omega\,|dz|)\) if
\[
\Big(\int_{\Gamma}\big|\mathcal{C}f(z)|^2 \omega(z)\,|dz|\Big)^{1/2}\le 
c \Big(\int_{\Gamma}|f(z)|^2 \omega(z)\,|dz|\Big)^{1/2}
\]
for all \(f\in C^\infty(\Gamma)\).
The norm \(\lVert \mathcal{C}\rVert_{\Gamma,\omega}\) 
is the infimum over admissible constants in this upper bound. 

The study of boundedness properties of the Cauchy transform and related singular integrals
is a hugely important branch of harmonic analysis and geometric measure theory.
By now, a rather complete theory is available. We will need the following result.

\begin{lem}
\label{lem:Cauchy-bd}
The norm of \(\mathcal{C}\) as an operator on \(L^2(\gamma_n,\omega_n\,|dz|)\),
satisfies
\[
\big\lVert \mathcal{C}\big\rVert_{\gamma_n,\omega_n} 
\lesssim \big[\omega_n\big]_{A_2,\gamma_n},
\]
as \(n\to\infty\).
\end{lem}

This is a deep fact, based on a remarkable connection between 
boundedness of the Cauchy transform, Carleson curves and Muckenhoupt 
weights. A composed locally-rectifiable curve $\Gamma$ is called 
a Carleson curve (or Ahlfors--David regular) if the arc-length of 
$\Gamma\cap B(x,\varepsilon)$ is comparable to $\varepsilon$ 
for every $x\in\Gamma$ and $0<\varepsilon<\rm{diam}(\Gamma)$. 
Since the lower bound is automatic for curves, this is equivalent to 
the existence of $C_\Gamma>0$ such that 
$|\Gamma\cap B(x,\varepsilon)|\leq C_\Gamma\varepsilon$ for 
$x\in \Gamma$ and $\varepsilon>0$. The constant $C_\Gamma$ is sometimes 
called the Carleson (or Ahlfors) constant of the curve $\Gamma$. 
Building on contributions from many different authors, 
boundedness of the Cauchy transform was characterized 
nearly 40 years ago, summarized in the following theorem.
\begin{thm}[\cite{DynkinEncycl}*{Theorem~5.15}, \cite{BK}*{Theorem~4.15}]
\label{thm:cauchy-bound}
Let $\Gamma$ be a composed locally-rectifiable curve. 
Then, the Cauchy transform $\mathcal{C}_\Gamma$ 
generates a bounded operator on $L^2(\Gamma,\omega|dz|)$ 
if and only if $\Gamma$ is a Carleson curve and 
$\big[\omega\big]_{A_2,\Gamma}<\infty$.
\end{thm}

We refer to \cite{BK}*{Section 4.4} for a historical account, and to \cite{Verdera}
for an excellent survey outlining the development of key ideas around Cauchy integrals.

We cannot immediately conclude Lemma~\ref{lem:Cauchy-bd} 
from Theorem~\ref{thm:cauchy-bound} since we want explicit 
control of the norm of the Cauchy transform on changing 
curves $\gamma_n$ with diverging Muckenhoupt characteristic 
$\big[\omega_n\big]_{A_2,\gamma_n}\asymp \log n$. However,
there are more precise results, which make the dependence of the norm of
the Cauchy integral on the weight explicit. First, by a version of
the \(A_2\)-theorem due to Lorist \cite{Lorist}*{Theorem~6.1}, we have that
\[
\big\lVert \mathcal{C}\big\rVert_{\gamma_n,\omega_n}
\le C_0\big[\omega_n\big]_{A_2,\gamma_n}\,
\big\lVert \mathcal{C}\big\rVert_{\gamma_n}
\]
where \(C_0\) is a constant which can be taken independent of 
\(n\)\footnote{More specifically, \(C_0\) depends only on the 
Carleson constants of \(\gamma_n\), the relevant homogeneous space constants, 
as well as on a Dini constant for the Cauchy kernel, which all remain uniformly 
bounded in our case} and 
$\big\lVert \mathcal{C}\big\rVert_{\gamma_n}
=\big\lVert \mathcal{C}\big\rVert_{\gamma_n,1}$ 
is the norm of the unweighted Cauchy transform on $\gamma_n$
(see also \cite{Anderson14}*{Theorem~1.1} and 
\cite{Lerner}*{Theorem~1.1} for related results).

To prove the lemma, it thus remains to establish
uniform bounds for the unweighted Cauchy transform on the curves \(\gamma_n\).
This follows from work by David \cites{David82,David}, 
who established that the Cauchy transform is bounded on Carleson curves. 
An essential ingredient in the proof is that the norm of the Cauchy 
transform on a Lipschitz graph can be estimated only in terms of the 
Lipschitz constant, proven by Coifman, McIntosh and Meyer \cite{CMM82}; 
see also \cites{CJS,Verdera}. By approximating the Carleson curve $\Gamma$ 
by Lipschitz graphs, David established the boundedness of $\mathcal{C}_\Gamma$. 
Carefully tracking the constants in the proof shows that the bound can be 
given in terms of the Carleson constant $C_\Gamma$. Since we know that 
the curves $\gamma_n$ are uniformly Carleson, i.e. 
$\sup_nC_{\gamma_n}<\infty$, this implies Lemma~\ref{lem:Cauchy-bd}. 
The content of \cites{David82,David} is presented 
in English in \cite{BK}*{Chapter 5}.

\subsection{Proof of the Marcinkiewicz--Zygmund inequality}
With the preparations in place, the proof is quite elegant.
We closely follow Chui--Zhong \cite{ChuiZhong}.

\begin{proof}[Proof of Lemma~\ref{lem:MZ}]
We apply Lemma~\ref{lem:interpolation} 
with \(a_{j}=P(z_j)\), which yields a function \(f\in H^2(D_n)\) with
\[
f(z_j)=P(z_j),
\]
satisfying the desired \(H^2\)-norm bound
\begin{equation}
\label{eq:interpolation-bd-f}
\int_{\gamma_n}|f(z)|^2 \,|dz|
\lesssim \sum_{j=0}^n |P(z_j)|^2 \,d(z_j,\gamma_n).
\end{equation}
By the triangle inequality,
\begin{equation}
\label{eq:int-P-to-int-f}
\int_{\gamma_n}|P(z)|^2\,|dz|\le 2\int_{\gamma_n}|f(z)|^2\,|dz|
+2\int_{\gamma_n}|f(z)-P(z)|^2\,|dz|,
\end{equation}
and hence we should estimate the norm of \(f-P\). 

To this end, note that \((f-P)/\mathcal E\) is holomorphic on \(D_n\),
so by Cauchy's theorem,
\begin{equation}
\label{eq:f-P}
f(z)-P(z)=\frac{\mathcal{E}(z)}{2\pi i}
\int_{\gamma_n}\frac{f(w)-P(w)}{\mathcal{E}(w)(w-z)}dw\\
=\frac{\mathcal{E}(z)}{2\pi i}
\int_{\gamma_n}\frac{f(w)}{\mathcal{E}(w)(w-z)}dw
\end{equation}
where the last step follows because \(\frac{P(w)}{\mathcal{E}(w)(w-z)}\) 
decays like
\(O(1/w^2)\) at infinity, which allows us to write
\[
\frac{1}{2\pi i}\int_{\gamma_n}\frac{P(w)}{\mathcal{E}(w)(w-z)}dw
=\lim_{R\to\infty}\frac{1}{2\pi i}\int_{|w|=R}\frac{P(w)}{\mathcal{E}(w)(w-z)}dw=0.
\]
We recall the notation \(\mathcal{C}h(z)\) for the Cauchy transform of a function \(h\)
along \(\gamma_n\), interpreted in the principal-value sense.
The above discussion shows that, on \(\gamma_n\), we have the identity
\[
f(z)-P(z)=\mathcal{E}(z)\big(\tfrac12 I+\mathcal{C}\big)(f/\mathcal{E})(z),
\]
where the contribution \(\frac12 I\) comes from Plemelj's formula \eqref{eq:Plemelj}. 
Hence the norm of \(f-P\) can be controlled by
\[
\int_{\gamma_n} |f(z)-P(z)|^2 \,|dz|\le 
\int_{\gamma_n} |\mathcal{E}(z)|^2 |
\big(\tfrac12I + \mathcal{C}\big)\big(f/\mathcal{E}\big)(z)|^2 \,|dz|.
\]
By Lemma~\ref{lem:Cauchy-bd}, the Cauchy transform is
bounded on \(L^2(\gamma_n,\omega_n\,|dz|)\) with norm 
\[
\lVert \mathcal{C}\rVert_{\gamma_n, \omega_n}
\lesssim [\omega_n]_{A_2,\gamma_n}.
\]
The same estimate clearly holds for \(\tfrac12I+\mathcal{C}\), and thus we have that
\[
\begin{aligned}
\int_{\gamma_n} |\mathcal{E}(z)|^2\, 
|\big(\tfrac12I+\mathcal{C}\big)\big(f/\mathcal{E}\big)(z)|^2 \,|dz|& \lesssim 
[\omega_n]^2_{A_2,\gamma_n}\int_{\gamma_n} |\mathcal{E}(z)|^2 
|f(z)/\mathcal{E}(z)|^2 \,|dz|\\
&=[\omega_n]^2_{A_2,\gamma_n}\int_{\gamma_n} |f(z)|^2 \,|dz|,
\end{aligned}
\]
where the implicit constant only depends on the fixed arc \(\gamma\).
Combining this estimate with \eqref{eq:int-P-to-int-f}, we get
\[
\int_{\gamma_n}|P(z)|^2\,|dz|\lesssim [\omega_n]_{A_2,\gamma_n}^2
\int_{\gamma_n} |f(z)|^2 \,|dz|
\]
which by the interpolation bound \eqref{eq:interpolation-bd-f} for \(f\) gives that
\[
\int_{\gamma}|P(z)|^2\,|dz|\lesssim [\omega_n]_{A_2,\gamma_n}^2 
\sum_{j=0}^n |P(z_j)|^2 d(z_j,\gamma_n).
\]
The claim now follows from the fact that
\(d(z_j,\gamma_n)\le C/n\).
\end{proof}

% !TeX root = ../main.tex
\section{Geometric lemmas}
\label{s:geom-lemmas}

\subsection{Growth of \texorpdfstring{\(A_2\)}{A2}-characteristics}
We next turn to the growth of the Muckenhoupt \(A_2\)-characteristics 
of the weights
$\omega_{n}=|\mathcal E|^2$, where $\mathcal E=\mathcal{E}_{n+1}$ 
is the polynomial with leading coefficient $\phi'(\infty)^n$ and 
zeroes at the points $z_j=z_{j,n}$ from Lemma~\ref{lem:extreme-pts}. 

\begin{prop}
\label{prop:A2-constant-E}
Let $\mathcal E(z)=\phi'(\infty)^n\prod_{j=0}^n (z-z_{j})$ and 
$\omega_{n} :=|\mathcal E|^2$. Then,
\begin{equation}
\label{eq:A2-constant-E-asymp}
[\omega_{n}]_{A_2,\gamma_n}\asymp \log n
\end{equation}
\end{prop}

We begin by reducing this to estimating the Muckenhoupt characteristic of 
the elementary weight $|z^2-1|$ on $\gamma_n$.
\begin{lem}\label{lem:omega-rough}
For $n\in \N$ big enough, it holds that
\begin{align*}
\omega_{n}(z)\asymp |z^2-1|\quad\text{for }\;z\in \gamma_n.
\end{align*}
\end{lem}

\begin{proof}
We first reduce the claim to the corresponding estimate for
the Faber approximation \(\mathcal E_0\) of 
\(\mathcal E\) from Section~\ref{ss:derivative}. 
By the proof of Lemma~\ref{lem:E-prim}, the interior zeros
\(w_j\) of \(\mathcal E_0\) satisfy
\[
w_j=z_j+O(e^{-cn}),
\]
while the zeros at the endpoints are left unchanged. Using
\(d(z_j,\gamma_n)\gtrsim n^{-2}\), and arguing as in the product
comparison at the end of that proof, we obtain
\[
\frac{\mathcal E_0(z)}{\mathcal E(z)}
=
\prod_{j=0}^n\frac{z-w_j}{z-z_j}
=1+O(e^{-cn}),
\qquad z\in\gamma_n,
\]
after possibly adjusting \(c>0\). It therefore suffices to prove that
\begin{equation}
\label{eq:E0-green-comparison}
|\mathcal E_0(z)|^2\asymp |z^2-1|,
\qquad z\in\gamma_n.
\end{equation}

To this end, let \(f\) be the dual Faber weight used to generate \(\mathcal{E}_0\),
and recall from Lemma~\ref{lem:Faber-dual-asymp} that
\(F=f\circ\varphi\) extends to a zero-free holomorphic function on a
neighborhood of \(\overline D\). Writing
\(\sigma=\varphi^{-1}(z)\), Theorem~\ref{thm:dual-Faber-asymp}
gives
\begin{equation}
\label{eq:E0-on-green-curve}
\mathcal E_0(z)=\sqrt{z^2-1}
\big(F(\sigma)\Phi(\sigma)^n-F(\sigma^{-1})\Phi(\sigma^{-1})^n\big)
+O\bigl(|z^2-1|r^n\bigr).
\end{equation}
The modulus matching condition \eqref{eq:modulus-match} says that
\[
|F(\sigma)|=|F(\sigma^{-1})|, \qquad \sigma\in\Gamma=\partial D,
\]
so since \(\Gamma_n=\varphi^{-1}(\gamma_n)\) converges to \(\Gamma\),
it follows that
\[
\frac{|F(\sigma^{-1})|}{|F(\sigma)|}=1+o(1)
\]
uniformly for \(\sigma\in\Gamma_n\).
For \(z\in\gamma_n\), we also have that \(|\Phi(\sigma)|=1+\frac1n\)
while \(|\Phi(\sigma^{-1})|\leq1\).
Hence we have that
\[
\bigg|\frac{F(\sigma)\Phi(\sigma)^n}{F(\sigma^{-1})\Phi(\sigma^{-1})^n}\bigg|
\ge \left(1+\frac1n\right)^{n}(1+o(1))=e+o(1),
\]
and so the first term in \eqref{eq:E0-on-green-curve}
dominates. Since its modulus is comparable to \(1\), 
we conclude from \eqref{eq:E0-on-green-curve} that
\[
|\mathcal E_0(z)|\asymp\sqrt{|z^2-1|},
\qquad z\in\gamma_n,
\]
which proves \eqref{eq:E0-green-comparison} and hence the lemma.
\end{proof}

The curves $\gamma_n$ degenerate as $n\to\infty$. 
For that reason, we again work with the transformation 
$\varphi(\sigma)=(\sigma+\sigma^{-1})/2$. 
Define $\Gamma=\partial D$ and put
\[
\Gamma_n=\varphi^{-1}(\gamma_n)=\Psi\bigl(\{w:|w|=1+1/n\}\bigr),
\] 
where \(\Psi=\Phi^{-1}\) is the inverse of the exterior conformal mapping of
the opened up arc.
The restriction of \(\varphi\) to \(\Gamma_n\) is a homeomorphism
onto \(\gamma_n\), so every subarc \(I\subset\gamma_n\) can be
written as \(I=\varphi(J)\) for a subarc \(J\subset\Gamma_n\).
By Lemma~\ref{lem:omega-rough}, for \(p\in\{-1,0,1\}\),
\begin{align}
\int_{\varphi(J)} \omega_{n}(z)^p |dz| &\asymp 
\int_{\varphi(J)} |z^2-1|^p |dz| 
= \int_J |\varphi(\sigma)^2-1|^p |\varphi'(\sigma)||d\sigma|\\
& =\int_J \frac{2}{(2|\sigma|)^{2p+2}} |\sigma^2-1|^{2p+1} |d\sigma|
\asymp\int_J  |\sigma^2-1|^{2p+1} |d\sigma|.
\end{align}
The last comparison is uniform in \(n\), since \(\Gamma_n\to\Gamma\) 
and \(\Gamma\) is bounded away from the origin. Plugging this into
the definition \eqref{eq:def-A2-constant}, we obtain
\begin{equation}
\label{eq:A2-opened}
[\omega_{n}]_{A_2,\gamma_n} \asymp \sup_{J\subset\Gamma_n}
\frac{\left(\int_J|\sigma^2-1|^3\,|d\sigma|\right)
\left(\int_J|\sigma^2-1|^{-1}\,|d\sigma|\right)}{\left(\int_J|\sigma^2-1|\,
|d\sigma|\right)^2}.
\end{equation}
The critical case, when this expression grows the fastest, 
comes from subarcs approaching either of the
points \(\pm1\). We first consider subarcs contained in a fixed
small neighborhood of one of these points.

\begin{lem}
\label{lem:local-param}
Fix \(\theta_0\) such that
\(\Psi(e^{i\theta_0})\in\{-1,1\}\), and let
\[
\sigma_n(x) = \Psi\Big(\big(1+\tfrac{1}{n}\big)e^{i(x+\theta_0)}\Big),
\qquad -\pi\leq x\leq\pi.
\]
There exists \(\delta>0\) such that, for
\(q\in\{-1,1,3\}\) and every subarc
\[
J\subset
\Gamma_n\cap B\bigl(\Psi(e^{i\theta_0}),\delta\bigr),
\]
we have
\[
\int_J|\sigma^2-1|^q\,|d\sigma| \asymp 
\int_{\sigma_n^{-1}(J)}\Big(x^2+\frac1{n^2}\Big)^{\frac{q}{2}}dx.
\]
The implied constants are independent of both \(J\) and \(n\).
\end{lem}

\begin{proof}
Since the argument is identical for both end-points, we may assume that 
\(\Psi(e^{i\theta_0})=1\). Put \(\zeta_0=e^{i\theta_0}\). Since \(\Psi\) is
conformal in a neighborhood of the unit circle,
\[
|\Psi(\zeta)-1|\asymp|\zeta-\zeta_0|,
\]
and \(|\Psi'(\zeta)|\asymp1\) for \(\zeta\) sufficiently close to \(\zeta_0\).

Let us calculate $|\zeta -\zeta_0|$ for 
$\zeta=\zeta(x) = (1+\frac{1}{n})e^{i (x+\theta_0)}$. 
By elementary trigonometry,
\[
|\zeta-\zeta_0|^2=2\Big(1+\frac1n\Big)(1-\cos x)+\frac1{n^2}
\asymp x^2+\frac1{n^2}
\]
for \(|x|\) sufficiently small. It follows that
\[
|\sigma_n(x)-1| \asymp \Big(x^2+\frac1{n^2}\Big)^{1/2}.
\]
After choosing \(\delta>0\) sufficiently small, we also have
\[
|\sigma_n(x)^2-1|
\asymp|\sigma_n(x)-1|
\]
whenever \(\sigma_n(x)\in B(1,\delta)\). Finally,
\[
|\sigma_n'(x)| = \Big(1+\frac1n\Big)
\Big|\Psi'\Big(\big(1+\tfrac1n\big)e^{i(x+\theta_0)}\Big)\Big| \asymp1,
\]
so the result now follows by changing variables
\(\sigma=\sigma_n(x)\).
\end{proof}

This allows us to estimate the ratio in \eqref{eq:A2-opened} 
for subarcs $J\subset B(\pm 1,\delta)$.

\begin{lem}\label{lem:A2-local-estimate}
For $\delta>0$ sufficiently small, we have that  
\begin{align*}
\sup_{J\subset \Gamma_n\cap B(\pm 1,\delta)}
\frac{\left(\int_J |\sigma^2-1|^3 |d\sigma|\right)
\left(\int_J |\sigma^2-1|^{-1} |d\sigma|\right)}{\left(\int_J |\sigma^2-1| 
|d\sigma|\right)^2} \asymp \log n.
\end{align*}
\end{lem}

\begin{proof}
Let $J\subset \Gamma_n\cap B(\pm 1,\delta)$ and 
$\sigma_n^{-1}(J)=[x_0,x_1]$. By Lemma~\ref{lem:local-param},
the ratio in the statement is comparable to
\begin{equation}
\label{eq:ratio-rewrite}
\frac{\left(\int_{x_0}^{x_1}
\left(x^2+\frac{1}{n^2}\right)^{3/2}dx\right)
\left(\int_{x_0}^{x_1}
\left(x^2+\frac{1}{n^2}\right)^{-1/2}dx\right)}
{\left(\int_{x_0}^{x_1}
\left(x^2+\frac{1}{n^2}\right)^{1/2}dx\right)^2}.
\end{equation}
Since the integrands are even, we may, up to absolute
multiplicative constants, suppose that
\[
0\leq x_0<x_1.
\]

For $0\leq x\leq x_1-x_0$, equivalence of the
\(\ell^1\)- and Euclidean norms on \(\mathbb R^2\) gives
\[
\sqrt{(x_0+x)^2+\frac{1}{n^2}}
\asymp \sqrt{x_0^2+\frac{1}{n^2}}+x.
\]
Put
\[
y_1=\frac{x_1-x_0}{\sqrt{x_0^2+n^{-2}}}.
\]
Performing the change of variables \(x=x_0+t\) followed by
\[
t=\sqrt{x_0^2+\frac1{n^2}}\,y,
\]
we obtain that
\begin{align}
\int_{x_0}^{x_1}\Big(x^2+\frac1{n^2}\Big)^{q/2}dx
& \asymp\int_0^{x_1-x_0}\bigg(\sqrt{x_0^2+\frac1{n^2}}+t\bigg)^q dt \\
&= \Big(x_0^2+\frac1{n^2}\Big)^{(q+1)/2} \int_0^{y_1}(1+y)^qdy.
\end{align}
In particular,
\begin{equation}
\int_{x_0}^{x_1}\Big(x^2+\frac{1}{n^2}\Big)^{3/2}dx
\asymp \Big(x_0^2+\frac{1}{n^2}\Big)^2 \int_0^{y_1}(1+y)^3dy,
\end{equation}
while
\begin{equation}
\int_{x_0}^{x_1}\Big(x^2+\frac{1}{n^2}\Big)^{-1/2}dx
\asymp \int_0^{y_1}(1+y)^{-1}dy,
\end{equation}
and
\begin{equation}
\int_{x_0}^{x_1}\Big(x^2+\frac{1}{n^2}\Big)^{1/2}dx
\asymp \Big(x_0^2+\frac{1}{n^2}\Big)\int_0^{y_1}(1+y)dy.
\end{equation}
Thus, the ratio \eqref{eq:ratio-rewrite} is comparable to
\[
\frac{\big((1+y_1)^4-1\big)\log(1+y_1)}{\big((1+y_1)^2-1\big)^2}.
\]
This expression is bounded for $0<y_1\leq1$ and comparable to
$\log(1+y_1)$ for $y_1\geq1$. Since 
\[
y_1 =\frac{x_1-x_0}{\sqrt{x_0^2+n^{-2}}}
\leq n(x_1-x_0) \leq \pi n,
\]
this proves
the upper bound.

For the reverse inequality, choose $c>0$ sufficiently small.
For $n$ large enough, there is a subarc
$J\subset\Gamma_n\cap B(\pm1,\delta)$ such that
\(\sigma_n^{-1}(J)=[0,c]\). The corresponding value of \(y_1\)
is \(cn\), and hence the ratio is comparable to \(\log n\).
\end{proof}

We can now finish the proof of Proposition~\ref{prop:A2-constant-E}.
\begin{proof}[Proof of Proposition~\ref{prop:A2-constant-E}]
Let $\delta>0$ be as in Lemma~\ref{lem:local-param}. 
Since all curves $\Gamma_n$ are uniformly Carleson, 
we can find $0<\delta_0<\delta/2$ such that the arc-length 
of each of the sets $\Gamma_n\cap B(\pm1,\delta_0)$ is at most
$\delta/8$.

Let $J\subset\Gamma_n$ be a subarc. We first suppose that
$|J|\geq\delta/2$. Then the portion of $J$ outside of
\[
B(1,\delta_0)\cup B(-1,\delta_0)
\]
has arc-length at least $\delta/4$. Since
$|\sigma^2-1|\geq\delta_0^2$ there, it follows that
\[
\int_J|\sigma^2-1||d\sigma|
\geq\frac{\delta}{4}\cdot \delta_0^2.
\]
Moreover, $|\sigma^2-1|$ and the lengths of the curves $\Gamma_n$
are uniformly bounded, and hence
\[
\int_J|\sigma^2-1|^3|d\sigma|\lesssim1.
\]
Away from $B(1,\delta_0)\cup B(-1,\delta_0)$, we also have
\[
|\sigma^2-1|^{-1}\leq\delta_0^{-2},
\]
while near either of the endpoints, Lemma~\ref{lem:local-param} gives
\begin{align*}
\int_{J\cap B(\pm1,\delta_0)}|\sigma^2-1|^{-1}|d\sigma|
& \lesssim \int_{-\pi}^{\pi} \left(x^2+\frac{1}{n^2}\right)^{-1/2}dx\\
&\lesssim\log n.
\end{align*}
Combining these estimates, we obtain
\[
\frac{\left(\int_J |\sigma^2-1|^3 |d\sigma|\right)
\left(\int_J |\sigma^2-1|^{-1} |d\sigma|\right)}{
\left(\int_J |\sigma^2-1| |d\sigma|\right)^2}
\lesssim \log n
\]
whenever $|J|\geq\delta/2$.

Suppose now that $|J|<\delta/2$. If
\[
J\cap\bigl(B(1,\delta_0)\cup B(-1,\delta_0)\bigr)=\emptyset,
\]
then $|\sigma^2-1|$ is bounded above and below on $J$, and the
ratio of integrals is bounded by a constant. If instead
$J\cap B(\pm1,\delta_0)\neq\emptyset$, then
\[
J\subset B(\pm1,\delta),
\]
since $\delta_0<\delta/2$ and $|J|<\delta/2$. Thus,
Lemma~\ref{lem:A2-local-estimate} applies and again gives the
upper bound $O(\log n)$.

We have therefore proved that
\[
\sup_{J\subseteq\Gamma_n}
\frac{\left(\int_J |\sigma^2-1|^3 |d\sigma|\right)
\left(\int_J |\sigma^2-1|^{-1} |d\sigma|\right)}
{\left(\int_J |\sigma^2-1| |d\sigma|\right)^2}
\lesssim\log n.
\]
The reverse inequality follows directly from
Lemma~\ref{lem:A2-local-estimate}. The result now follows from
\eqref{eq:A2-opened}.
\end{proof}

\subsection{Separation of extreme points}
\label{s:AppB-separation}
We conclude by verifying the necessary separation 
properties of the points \(z_j=z_{j,n}\),
needed to apply the Marcinkiewicz--Zygmund inequality of Chui--Zhong.
\begin{lem}
\label{lem:separation}
Let \(z_0,\dotsc,z_n\) be the extremal points from
Lemma~\ref{lem:extreme-pts}, and let
\[
\gamma_n=\bigl\{z\in\C:|\phi(z)|=1+1/n\bigr\}.
\]
Then, for \(j=1,\ldots,n\), we have that
\[
|z_j-z_{j-1}|\asymp d(z_j,\gamma_n).
\]
\end{lem}

\begin{proof}
We first establish the asymptotic equality
\begin{equation}
\label{eq:refined-spacing}
|z_j-z_{j-1}|\asymp\frac{1+\min\{j,n-j\}}{n^2}
\end{equation}
for \(j=1,\ldots,n\).
Below, we will also prove that
\begin{equation}
\label{eq:distance-green-curves}
d(z_j,\gamma_n)\asymp\frac{1+\min\{j,n-j\}}{n^2},
\end{equation}
and the assertion of the lemma follows by combining these two claims.

Note that it suffices to establish \eqref{eq:refined-spacing} 
for small indices \(j\), i.e., such that \(j/n\) is bounded away from \(1\).
This amounts to a local analysis around the end-point \(z=-1\) of the arc.
The case of \(j\) close to \(n\) will then follow in precisely the same way
by an analogous analysis around the remaining end-point.

We keep the notation from the proof of
Lemma~\ref{lem:extreme-pts}, in particular \(\theta\) denotes the
argument of \(\phi_+/\phi_-\) and \(\beta\) the argument of \(g_+/g_-\).
The latter has uniformly bounded arc-length derivative \(\partial_s \beta\)
and satisfies \(\beta(1)=\beta(-1)=0\).
The points \(z_j\) are determined by the equations
\[
\theta(z_j) + \frac{\beta(z_j)}{n}=\frac{2\pi j}{n},\qquad j=0,\ldots,n.
\]
Equivalently, \(z_j\) can be seen as the \(j/n\)-th quantile of the measure
\(\mu_n=f_n(z)|dz|\) on \(\gamma\) with density
\[
\begin{aligned}
f_n(z)&=\frac{\partial_s \theta(z)}{2\pi} + \frac{\partial_{s}\beta(z)}{2\pi n}.
\end{aligned}
\]
The function \(\theta\) is \(2\pi\) times the cumulative distribution function for 
the harmonic measure \(\omega_{\gamma,\infty}\) (see, e.g., \cite{CJV}*{Section~2.2}). 
Combined with \cite{Totik2014_2}*{Appendix A}, this gives that
\[
\frac{\partial_s \theta(z)}{2\pi}=\frac{d\omega_{\gamma,\infty}(z)}{|dz|}
\asymp |z^2-1|^{-1/2},
\]
and since \(\partial_s \beta\) is bounded the same estimate holds for \(f_n(z)\). 
Let \(\zeta:[0,\tau]\to \gamma\) be an arc-length parametrization of 
\(\gamma\) with \(\zeta(0)=-1\). Then \(\zeta\) is regular, in particular
\[|\zeta(s)-\zeta(t)|\asymp |s-t|, \qquad s,t\in [0,\tau].\]
With \(z_j=\zeta(y_j)\) we see that
\[
\frac{1}{n}=\int_{y_{j-1}}^{y_j}f_n(\zeta(s))ds
\asymp \int_{y_{j-1}}^{y_j} \frac{ds}{\sqrt{s}}=2\big(\sqrt{y_j}-\sqrt{y_{j-1}}\big),
\]
where we have used that \(z_j\) is bounded away from \(1\). 
Then, summing over \(j\) and using that \(y_0=0\), we find that
\[
\sqrt{y_j}\asymp \frac{j}{n}.
\]
and as a consequence 
\[
|z_{j}-z_{j-1}|\asymp y_{j}-y_{j-1}
=(\sqrt{y_j}-\sqrt{y_{j-1}})(\sqrt{y_j}+\sqrt{y_{j-1}})
\asymp \frac{1+j}{n^2}
\]
for \(j/n\) bounded away from \(1\), as required.

For use below, let us also note that
\[
\frac{1}{n}=\int_{y_{j-1}}^{y_j}f_n(\zeta(s))ds \asymp 
\int_{y_{j-1}}^{y_j}\frac{\partial_s\theta(\zeta(s))}{2\pi}ds 
= \theta(z_j)-\theta(z_{j-1}) + O(j/n^2),
\]
and summing over \(j\) we obtain
\begin{equation}
\label{eq:theta_z_j_estimate}
\theta(z_j)\asymp \frac{2\pi j}{n}.
\end{equation}

We proceed to estimate the distance \(d(z_j,\gamma_n)\) 
from \(z_j\) to the Green curve
\(\gamma_n=\big\{\phi^{-1}(w):|w|=1+1/n\big\}\), 
which can be written as
\[
d(z_j,\gamma_n)=\inf_{|w|=1+1/n}\big|z_j-\phi^{-1}(w)\big|.
\]
We claim that, uniformly for \(z\in\gamma\) 
and \(w\) in a fixed neighborhood of the unit circle, 
the expression on the right-hand satisfies the two-sided estimate
\begin{equation}
\label{eq:asymp-equiuv-z-phiw}
\big|z-\phi^{-1}(w)\big| \asymp \big|w-\phi_+(z)\big|\,\big|w-\phi_{-}(z)\big|.
\end{equation}
To see why this holds, we recall from Figure~\ref{fig:maps} 
the maps \(\Phi\) and \(\varphi\) and their 
relation to the exterior conformal map \(\phi\),
\[
\varphi=\phi^{-1}\circ\Phi.
\]
Recall also that the map \(\Phi\) is conformal 
in a neighborhood of the opened up arc \(\Gamma=\varphi^{-1}(\gamma)\). 
Writing \(w=\Phi(\xi)\) 
and \(\phi_+(z)=\Phi(\sigma)\), we find that
\[
\phi^{-1}(w)-z=\phi^{-1}(\Phi(\xi))-\phi^{-1}(\Phi(\sigma))
=\varphi(\xi)-\varphi(\sigma).
\]
Furthermore, it follows from the definition 
of the Joukowski map \(\varphi\) that
\[
\varphi(\xi)-\varphi(\sigma)
=\frac{(\xi-\sigma)(\xi-\sigma^{-1})}{2\xi}.
\]
Since \(\Phi\) is conformal on a neighborhood 
of the analytic curve \(\Gamma\), 
this gives \eqref{eq:asymp-equiuv-z-phiw}.

We next record the estimate
\begin{equation}
\label{eq:circle-product-dist}
\inf_{|w|=1+1/n}|w-a|\,|w-b|\asymp\frac1n\Big(\frac1n+|a-b|\Big),
\qquad a,b\in\mathbb{T}.
\end{equation}
Indeed, both factors on the left are bounded below 
by \(1/n\), and by the triangle inequality one of 
them is at least \(|a-b|/2\). Hence, the lower bound follows.
For the reverse inequality, it suffices to test with
\(w=(1+1/n)a\). 

To complete the proof, it suffices to note 
that Equations \eqref{eq:asymp-equiuv-z-phiw} and
\eqref{eq:circle-product-dist} give
\[
d(z_j,\gamma_n)\asymp\frac1n\Big(\frac1n+|\phi_+(z_j)-\phi_-(z_j)|\Big).
\]
Hence the claim will follow if we show that
\begin{equation}
\label{eq:phiplus-phiminus}
|\phi_+(z_j)-\phi_-(z_j)|\asymp \frac{\min\{j,n-j\}}{n}.
\end{equation}
Since \(\phi_\pm(z_j)\) is unimodular, we have that
\[
\begin{aligned}
|\phi_+(z_j)-\phi_-(z_j)|&=\bigg|\frac{\phi_+(z_j)}{\phi_-(z_j)}-1\bigg|
=2\Big|\sin\Big(\frac{\theta(z_j)}{2}\Big)\Big|.
\end{aligned}
\]
The claim now follows from \eqref{eq:theta_z_j_estimate}.
\end{proof}

\subsection*{Acknowledgements}
We thank Athanasios Kouroupis for stimulating discussions 
throughout the work on this paper. We are grateful to 
Joaquim Ortega-Cerd\`{a}, Stefanie Petermichl, Joan Verdera 
and Andrei Lerner for very helpful advice on various aspects 
of harmonic analysis.

The first author is supported by Grant 1154426N from Research Foundation Flanders (FWO). 
The research of the third author was supported by Odysseus Grant G0DDD23N from 
Research Foundation Flanders (FWO) and from
the G\"{o}ran Gustafsson Foundation for Research in Natural Sciences and Medicine.
The fourth author received support from Odysseus Grant G0DDD23N 
from Research Foundation Flanders (FWO).

\subsection*{Disclosure of AI use}
All the main ideas and results of this paper are due to the authors. 
AI tools were used only in the final stages of the project to 
help with polishing and proof reading.

% !TeX root = ../main.tex

\bigskip
\bigskip
\bigskip

\noindent 
\begin{minipage}[t]{0.54\textwidth}
\noindent \sc Benedikt Buchecker\newline
KU Leuven\newline
Leuven, Belgium
\newline {\tt benedikt.buchecker@kuleuven.be}
\end{minipage}
\hfill
\begin{minipage}[t]{0.5\textwidth}
\noindent \sc Benjamin Eichinger\newline
TU Wien \newline
Vienna, Austria
\newline {\tt benjamin.eichinger@tuwien.ac.at}
\end{minipage}
\\ 
\medskip
\bigskip

\noindent 
\begin{minipage}[t]{0.54\textwidth}
\noindent \sc Olof Rubin\newline
Royal Institute of Technology\newline
Stockholm, Sweden
\newline {\tt orubin@kth.se}
\end{minipage}
\hfill
\begin{minipage}[t]{0.5\textwidth}
\noindent \sc Aron Wennman\newline
KU Leuven\newline
Leuven, Belgium\newline
{\tt aron.wennman@kuleuven.be}
\end{minipage}
\end{document}